\documentclass[11pt]{article}

\usepackage{amsmath,amsfonts,amssymb,amsthm}
\usepackage{fullpage,subfiles}

\usepackage{graphicx}
\usepackage[colorlinks=true, allcolors=blue]{hyperref}

\newtheorem{theorem}{Theorem}[section]
\newtheorem{corollary}[theorem]{Corollary}
\newtheorem{lemma}[theorem]{Lemma}

\theoremstyle{definition}

\newtheorem{remark}[theorem]{Remark}

\newcommand{\expect}[2]{\mathbb{E}_{#1} \left\{ {#2} \right\}}
\newcommand{\prob}[2]{\mathbb{P}_{#1} \left[ {#2} \right]}

\newcommand{\AND}{\quad \text{and}\quad}

\newcommand{\R}{\mathbb{R}}

\newcommand{\mydet}[1]{\mathrm{det}\left[ {#1} \right]}
\newcommand{\mytr}[1]{\mathrm{Tr}\left[ {#1} \right]}

\newcommand{\kk}{k}
\newcommand{\dd}{d}

\renewcommand{\ss}{n_1}
\renewcommand{\tt}{n_2}

\newcommand{\trans}{\intercal}

\newcommand{\MM}{\mathcal{M}}

\newcommand{\SPM}{\mathcal{SPM}}

\newcommand{\nn}{n}
\newcommand{\mm}{m}

\newcommand{\applyto}[1]{\left\{ {#1} \right\}}

\newcommand{\arxiv}[1]{\href{https://arxiv.org/abs/{#1}}{arXiv: {#1}}}

\title{Finite free point processes}
\author{Adam W. Marcus
\thanks{Research supported by NSF CAREER Grant DMS-1552520.}
\\ \'Ecole Polytechnique F\'ed\'erale de Lausanne
}

\graphicspath{{./}{pics/}}
\begin{document}

\maketitle

\begin{abstract}

\noindent We use techniques from finite free probability to analyze matrix 
processes 
related to eigenvalues, singular values, and generalized singular values of 
random matrices.
The models we use are quite basic and the analysis consists entirely of 
expected characteristic polynomials.
A number of our results match known results in random matrix theory, however 
our main result (regarding generalized singular values) seems to be more 
general than any of the standard random matrix processes 
(Hermite/Laguerre/Jacobi) in the field.
To test this, we perform a series of simulations of this new process that, on 
the one hand, confirms that this process can exhibit behavior not seen in the 
standard random matrix processes,
but on the other hand provides evidence that the true behavior is captured 
quite well by our techniques.
This, coupled with the fact that we are able to compute the same statistics 
for this new model that we are for the standard models, suggests that further 
investigation could be both interesting and fruitful.

\end{abstract}

\section{Introduction}

For integers $m, n$, let $\MM_{m, n}$ denote the collection of real $m \times 
n$ matrices and let $X_{m \times n} \in \MM_{m, n}$ have entries consisting of
independent, identically distributed symmetric\footnote{$\prob{}{Y} = 
\prob{}{-Y}$.} random variables with variance 
$1$ (and all other moments bounded).
We consider the following questions:
\begin{description}
\item[1. Eigenvalues:] If $A \in \MM_{n, n}$, how do the eigenvalues of 
\[
A + \sqrt{\theta} X_{n \times n}
\]
evolve for $\theta \geq 0$?

\item[2. Hermitian Eigenvalues:] If $B \in \MM_{n, n}$, how do the eigenvalues 
of 
\[
(B + \sqrt{\theta} X_{n \times n}) + (B + \sqrt{\theta} X_{n \times n})^\trans
\]
evolve for $\theta \geq 0$?
\item[3. Singular Values:]  
If $C \in \MM_{m, n}$, how do the (squares of the) singular values of 
\[
C + \sqrt{\theta} X_{m, n}
\]
evolve for $\theta \geq 0$?
\item[4. Generalized Singular Values:] If $D_1 \in \MM_{\ss, \kk}$ and $D_2 \in 
\MM_{\tt, \kk}$, how do the (squares of the) generalized singular values of the 
pair
\[
\left\{ D_1 + \sqrt{\theta} X_{\ss \times \kk}, D_2 + \sqrt{\theta} X_{\tt 
\times \kk} 
\right\}
\]
evolve for $\theta \geq 0$?
\end{description}

In this paper we investigate topics \textbf{2.}, \textbf{3.}, \textbf{4.}
using techniques from finite free probability\footnote{Topic \textbf{1.} has 
the added complication of having a decomposition whose matrices 
may or may not be unitary, which is something that 
must be addressed by alternative means (see 
\cite{benno} for developments in this direction).}.
Our aim is to show that these simple models and the information we can get 
about them connect deeply to the standard models in random matrix theory.

We will first show that \textbf{2.} is related to a problem that has 
already been written about by Tao \cite{tt}, so (in fact) the real focus in 
this paper will be on \textbf{3.} and \textbf{4.} and the majority of the work 
will lie in developing a tool that is robust enough to handle both cases at 
once (Section~\ref{sec:poly}).
Our main tool will be finite free probability --- or, in other words, expected 
characteristic polynomials.
Despite only containing a fraction of the information about a distribution, we 
will find that expected characteristic polynomials capture essentially all of 
the ``non-random'' elements of random matrix processes.
In particular, the point process derived from \textbf{2.} will reproduce the 
Hermite (or Gaussian) process (Section~\ref{sec:hermite}) and the point process 
derived from \textbf{3.} will reproduce the Laguerre (or Wishart) process 
(Section~\ref{sec:laguerre}).

While those familiar with random matrix processes might naturally expect the 
point process derived from \textbf{4.} to reproduce the Jacobi process, we find 
that is not the case (though the two are related --- see 
Section~\ref{sec:gsvdbm2}).
Most notably, the dynamics of \textbf{4.} depend very heavily on the 
initial matrix, a feature not present in the Hermite, Laguerre, and 
Jacobi processes.
As far as we can tell, the dynamics displayed by \textbf{4.} are (in this 
sense) more complex than any of the processes typically studied in random 
matrix theory, and so may suggest a fruitful direction of further research.

In order to get a better understanding of the point process derived from 
\textbf{4.}, we provide the results of a simulation that produces two vastly 
different behaviors from the same starting points and then compare this to the 
information provided by the expected characteristic polynomials.
This, along with some other experimental results, is presented in  
Section~\ref{sec:simulation}.
Finally, we end with some concluding comments and acknowledgments in 
Section~\ref{sec:conclusion}.

\subsection{Conventions}\label{sec:prelims}

As mentioned in the introduction, we will let $\MM_{m, n}$ denote the 
collection of $m \times n$ real matrices and will always use $X_{m \times n} 
\in 
\MM_{m, n}$ to denote a matrix with entries consisting of
independent, identically distributed symmetric random variables with variance 
$1$ (and all other moments bounded)\footnote{It is likely this condition can be 
weakened substantially, but that is not something we will be concerned with in 
this paper.}.
The choice of reals is arbitrary --- one could substitute the complex (and 
conjugate transpose for transpose) and get the same results (except for 
Section~\ref{sec:simulation}).

For an operator $L:\R[x_1, \dots, x_n] \to \R[x_1, \dots, x_n]$, we will write 
\[
L \applyto{p}
\]
to denote $L$ applied to the polynomial $p$.

We will use $\SPM_{\nn} \subseteq \MM_{\nn, \nn}$ to denote the set of signed 
permutation matrices of size $\nn$.
We will say that a random matrix $Y \in \MM_{m \times n}$ is {\em 
$\SPM$--bi-invariant} if 
\[
\prob{}{Y} 
= \prob{}{Q_{m} Y}
= \prob{}{Y Q_n} 
\] 
for $Q_m \in \SPM_m$ and all $Q_n \in \SPM_n$ and we will say that a random 
matrix $Z \in \MM_{n \times n}$ is {\em $\SPM$--conjugation-invariant} if 
\[
\prob{}{Y} =
\prob{}{Q_n Y Q_n^T}
\]
for all $Q_n \in \SPM_{n}$.
It should be clear that (by construction) $X_{m, n}$ is $\SPM$--bi-invariant 
for all $n, m$.
It is then easy to check that $X_{\nn, \nn} + X_{\nn, \nn}^\trans$ is 
$\SPM$-conjugation--symmetric.

Lastly, we note that the convention in random matrix theory is to consider the 
{\em squares} of the (generalized) singular values, as this can then be viewed 
as an ``eigenvalue'' problem and we will do the same (in order to compare 
results effectively).

\newcommand{\ddd}[1]{\,\mathrm{d}#1}

\subsection{Some Random Matrix Theory}\label{sec:RMT}

Here we list the relevant random matrix processes along with the known results 
concerning the evolution of their derived point processes.
In each case, the behavior of the point process differs depending on whether 
the Brownian motion involved is real valued ($\beta = 1$) or complex valued 
($\beta = 
2$) valued\footnote{Technically, quaternion valued ($\beta = 4$) is possible, 
but that presents complications we wish to avoid here.}.

\subsubsection{Hermite Process}
The real (complex) matrix Hermite process was introduced and studied by Dyson 
\cite{dyson}. 
It can be constructed as a Wiener process $(H_t)_{t \in 0, \infty}$ with $H_0$ 
a fixed symmetric (Hermitian) matrix and with each increment $H_t - H_s$ equal 
to $\sqrt{t-s}$ times a Gaussian orthogonal (unitary) matrix.
Assuming $H_0$ has distinct eigenvalues, one can show that the eigenvalues of 
$H_t$ form a point process $\lambda = \lambda(t)$ which is the unique strong 
solution to the stochastic differential equation (SDE)
\begin{equation}\label{eq:hermite_SDE}
\dd \lambda_i = \ddd{W_i} +
 \frac{\beta}{2} \left[ 
\sum_{j \neq i} \frac{1}{\lambda_i - \lambda_j}
\right] \ddd{t}
\end{equation}
where the $W_i$ are (1-dimensional) independent Brownian motions.

\subsubsection{Laguerre Process}

The real (complex) matrix Laguerre process --- also known as the Wishart 
process --- was introduced and studied in \cite{bru}.
It can be constructed as a Wiener process $(L_t)_{t \in 0, \infty}$ with $L_0$ 
a fixed $\mm \times \nn$ real (complex) matrix with $\mm \leq \nn$ and with 
each increment $L_t - 
L_s$ equal to an $\mm \times \nn$ matrix with entries independent real 
(complex) Gaussian random variables with variance $t - s$.  
%
%
Assuming $L_0$ has distinct singular values, the $\mm$ eigenvalues of 
$L_tL_t^\trans$ (so the squares of the singular values of 
$L_t$) 
then form a point process $\lambda = \lambda(t) \in \R^\mm$ which is the unique 
strong solution to the SDE
\begin{equation}\label{eq:laguerre_SDE}
\dd \lambda_i = 2 \sqrt{\lambda_i} \ddd{W_i} +
 \beta \left[ 
\nn
+ \sum_{j \neq i} \frac{\lambda_i + \lambda_j}{\lambda_i - \lambda_j}
\right] \ddd{t}
\end{equation}
where the $W_i$ are (1-dimensional) independent Brownian motions.

\subsubsection{Jacobi Process}

The real (complex) matrix Jacobi process was introduced and studied in 
\cite{doumerc}.
The method of construction given in \cite{doumerc} is to first construct a 
Brownian motion $Q_n$ on the group of $n \times n$ orthogonal (unitary) using a 
heat kernel as in \cite{biane}.
Then for fixed parameters $\mm, p, q \leq n$ with $q = \nn-p$, one forms the $m 
\times m$ 
matrix 
\[
J_m := P_m Q_n P_p^\trans P_p Q_n^\trans P_m^\trans 
\]
where $P_k$ denotes the $k \times n$ projection matrix
\[
P_k = \begin{bmatrix} I_k & 0_{k \times (n-k)} \end{bmatrix}.
\]
It was shown in \cite{doumerc} that when $J_\mm$ is real (complex) 
and $\max \{ p, q \} \geq m-1 + 2/\beta$ for $\beta =1$ ($\beta = 2$), the 
eigenvalue process $\lambda = \lambda(t)$ is the unique strong solution of 
the stochastic differential 
equation
\begin{equation}\label{eq:jacobi_SDE}
\dd \lambda_i = 2 \sqrt{\lambda_i (1-\lambda_i)} \ddd{W_i} +
 \beta \left[ p (1 - \lambda_i) -  
q\lambda_i
+ \sum_{j \neq i} \frac{\lambda_i(1-\lambda_j) + 
\lambda_j(1-\lambda_i)}{\lambda_i - \lambda_j}
\right] \ddd{t}
\end{equation}
where the $W_i$ are (1-dimensional) independent Brownian motions.

\begin{remark}
It should be noted that the construction of the Jacobi process is quite 
different from the previous two. 
We suspect that this is because the natural way to generalize the constructions 
of the Hermite and Laguerre processes to generalized singular values does not 
reduce to a simple SDE (this is discussed in Section~\ref{sec:gsvdbm}).
\end{remark}

%
%

\subsection{Some remarks}\label{sec:goals}

As we claimed earlier, the goal is to show (and attempt to exploit) a 
seemingly strange connection between the point processes discussed in the 
previous section and ones derived from expected characteristic polynomials. 
A common theme in random matrix theory is the fact that certain distributions 
depend on the number field over which the entries are drawn.
This dependence on the number field is typically denoted $\beta$ and the 
general goal is to express solutions as a function of $\beta$ in such a way 
that $\beta = 1, 2, 4$ corresponds to the real, complex, and quaternion number 
systems.
Once one has an object (say, for example, a distribution) expressed in terms 
of $\beta$ in this way, one can consider the collection of objects one would 
get by plugging in any $\beta \geq 0$ as a parameter.  
While these may not necessarily correspond to random matrix, understanding the 
dependence of the objects on the parameter $\beta$ (apart from being of 
interest on its own) can often give insights into the original random matrix 
problems corresponding to the specific cases $\beta = 1, 2, 4$.
The idea that random matrices represent three distinct points on a continuum 
of probability models was called the ``Threefold Way" by Dyson in his original 
paper.
The author shares a belief first put forth by Edelman\footnote{At least as far 
as the author knows.} that there is at least one other ``Way'' --- that is, a 
fourth distinct point on the continuum of probability models that can be 
modeled using random matrices.
This distinct ``point'' is the limit $\beta \to \infty$, with the 
corresponding model being the algebra of expected characteristic polynomials 
known as ``finite free probability". 

It is worth noting that a common interpretation of the $\beta$ parameter is as 
an inverse temperature (stemming from Dyson's original paper \cite{dyson} where 
he noted that a number of the distributions he derived would be similar to what 
one would get from an abstract ``log-gas'').
In that respect, one would expect that a model for $\beta \to \infty$ would 
consist of the ``non-random'' contributions to a given random process.
While we will not attempt to define ``non-random'' rigorously (in general), 
it is clear what this would mean in the specific context of the processes 
listed in Section~\ref{sec:RMT}.
Each of these processes is a {\em drift--diffusion process}, with the parts 
that integrate against $\ddd{t}$ being referred to as the ``drift'' and the 
parts that integrate against Brownian motions being referred to as the 
``diffusion.'' 
Hence in this respect, we should expect a $\beta \to \infty$ model to exhibit 
similar drift behavior to the $\beta = 1, 2$ models, but without any 
diffusion term.

\section{Hermitian eigenvalues }\label{sec:hermite}

We start by discussing the case of Hermitian eigenvalues.
Our main tool will be a theorem that first appeared (in spirit) in 
\cite{ff_main} but was proved in the generality of 
$\SPM$--conjugation-invariance in \cite{DUI}.

\begin{theorem}\label{thm:additive_conv}
Let $A, B \in \MM_{\nn, \nn}$ be independent random matrices with $B$ 
being $\SPM$--conjugation-invariant. 
Furthermore, let $P, Q$ be power series for which the operators $P(\partial)$ 
and $Q(\partial)$ satisfy 
\[
\expect{A}{\mydet{x I + A}} 
= P(\partial) \applyto{x^\nn} 
\AND
\expect{B}{\mydet{x I + B}} 
= Q(\partial) \applyto{x^\nn}. 
\]
Then
\[
\expect{A, B}{ \mydet{x I + A + B } } 
= (PQ)(\partial) \applyto{x^\nn} 
\]
where $PQ$ denotes the multiplication of the two power series.
\end{theorem}

As mentioned in Section~\ref{sec:prelims}, $X_{\nn, \nn} + X_{\nn, \nn}^\trans$ 
is $\SPM$-conjugation--invariant, and so Theorem~\ref{thm:additive_conv} gives 
us (in 
theory) a way to compute
\[
\expect{}{ \mydet{x I - (A + \sqrt{\theta} X_{\nn, \nn}) - (A + \sqrt{\theta} 
X_{\nn, 
\nn})^\trans }}
\]
for any $A \in \MM_{\nn, \nn}$ simply by knowing 
\[
\expect{A}{\mydet{x I - (A + A^\trans)}} 
\AND
\expect{}{\mydet{x I - \sqrt{\theta}(X_{\nn, \nn} + X_{\nn, \nn}^\trans})} 
\]
To do so, however, we will need to find a power series $Q$ which 
satisfies
\[
\expect{}{\mydet{x I - \theta(X_{\nn, \nn}) + X_{\nn, \nn}^\trans}) }
= Q(\partial) \applyto{x^\nn}
\]
This can be derived from a result of Edelman which showed that the expected 
characteristic polynomial of the Gaussian Unitary Ensemble is a Hermite 
polynomial \cite{edelman_thesis}.
We give an alternative derivation here:
\begin{lemma}\label{lem:hermite}
For all integers $\nn > 0$, we have
\[
\expect{}{\mydet{x I - \sqrt{\theta}(X_{\nn \times \nn} + X_{\nn \times 
\nn}^\trans) 
}} = 
e^{-\theta \partial^2} \applyto{ x^\nn }
\]
where we are writing $e^{-\theta \partial^2}$ as shorthand\footnote{This also 
has the benefit of revealing the underlying semigroup.} for the operator 
\[
e^{-\theta \partial^2} = \sum_{i=0}^\infty \frac{(-\theta)^i}{i!} \partial^{2i}.
\]
\end{lemma}
\begin{proof}
Let $Y = \frac{X_{\nn \times \nn} + X_{\nn \times \nn}^\trans}{\sqrt{2}}$.
This is a symmetric matrix with random variable entries, each of which has 
expectation $0$ and variance $1$.
Furthermore, all of the entries of $Y$ are independent except for the imposed 
symmetry $Y_{i, j} = Y_{j, i}$.  
We now consider the expansion of the polynomial
\begin{equation}\label{eq_Y}
\expect{}{\mydet{x I + Y}}
\end{equation}
using the formula
\[
\expect{}{\mydet{A}} = \sum_{\sigma \in S_n} (-1)^{\mathrm{sgn}(\sigma)} 
\expect{}{A_{1, 
\sigma(a)} \dots A_{n, \sigma(n)}}.
\]
Of particular note is the fact that the diagonal entries of $Y$ do not 
contribute in any way to \eqref{eq_Y} --- the expansion is affine in terms of 
each diagonal entry and so when the expectation is taken, these terms will all 
become $0$.
Hence we can replace the diagonal entries of $Y$ with (deterministic) $0$ 
entries and not change \eqref{eq_Y} --- we call the remaining matrix $Z$.

The utility of this transformation is that $Z$ can now be viewed as the 
weighted adjacency matrix of the complete graph with independent random 
variables on the edges, each of which has mean $0$ and variance $1$.
It is now a somewhat well-known result of Godsil and Gutmann that the expected 
characteristic polynomial of such an adjacency matrix is the matching 
polynomial of the graph \cite{gg}. 
It is also well known (see \cite{godsil}) that the matching polynomial of the 
complete graph on $n$ vertices is the Hermite polynomial, which has the formula 
\cite{szego}
\[
H_{\nn}(x) = e^{-\frac{\partial^2}{2}} \applyto{x^\nn}.
\]
From this, it is easy to check that
\[
\expect{}{\mydet{x I - t Z}} = t^\nn H_{\nn}(x/t) = 
e^{-t^2\frac{\partial^2}{2}} \applyto{x^\nn}.
\]
The theorem follows by setting $t = \sqrt{2\theta}$.
\end{proof}

A combination of Theorem~\ref{thm:additive_conv} and Lemma~\ref{lem:hermite} 
imply 
that for a given matrix $B \in \MM_{\nn, \nn}$ with 
\[
p_B(x) = \mydet{x I - B - B^\trans}
\]
we have the formula
\begin{equation}\label{eq:hermite}
\expect{}{\mydet{x I - (B + X_{\nn, \nn}) -  (B + X_{\nn, \nn})^\trans}}
= 
e^{-\theta \partial^2} \applyto{ p_B(x) }.
\end{equation}
At this point we could more or less refer to the reader to \cite{tt}, where an 
exploration of this process was conducted for other reasons (and contains a 
number of interesting observations that are beyond the scope of this paper).
However it will be instructive to outline relevant parts of the analysis 
briefly since we take a slightly different approach which will be somewhat 
instructive for the results in Section~\ref{sec:bm}.
To analyze \eqref{eq:hermite}, we start by finding a differential 
equation that it satisfies.
Let $p(x)$ be an arbitrary polynomial and let
\[
q(x, t) = e^{-\theta \partial^2} \applyto{ p(x) }.
\]
Now define the root path $r(\theta)$ so that $q(r(\theta), \theta) = 0$.
Taking a derivative in $\theta$, we have
\[
0 = \frac{\partial}{\partial \theta} q(r(\theta), \theta) 
= q_1(r(\theta), \theta) r'(\theta) + q_2(r(\theta), \theta)
\]
so (setting $\theta = 0$)
\[
0 = p'(r) r' - p''(r) 
\]
or
\[
r'(0) = \frac{p''(r)}{p'(r)}.
\]
Now if we set 
\[
p(x) = \prod_{i = 0}^\dd (x - \lambda_i)
\]
then at $x = \lambda_i$ we have
\begin{equation}\label{eq:deriv_trick}
\frac{p''(\lambda_i)}{p'(
\lambda_i)}
= \sum_{j \neq i} \frac{2}{\lambda_i - \lambda_j}.
\end{equation}
Hence evolution of the roots of \eqref{eq:hermite} satisfy the differential 
equations
\[
\frac{\partial \lambda_i}{\partial \theta}
= 
 \sum_{j\neq i} \frac{2}{\lambda_i - \lambda_j}.
\]
At this point in his presentation, Tao \cite{tt} notes ``Curiously, this system 
resembles that of 
Dyson Brownian motion (except with the Brownian motion part removed, and time 
reversed\footnote{His 
comment about time reversal stems from the fact that he was using the operator 
$e^{t \partial^2}$ instead of $e^{-t \partial^2}$.}).''
Specifically, he is noting that this evolution matches (up to a constant) the 
drift term in the Hermite process \eqref{eq:hermite_SDE}.

The ``curiosity'' comes from the fact that one would not expect this to be the 
case --- in some sense, this correspondence can be interpreted as asserting 
that a random variable $Y$ (in our case, the characteristic polynomial of a 
certain random matrix) satisfies
\[
f(\expect{}{Y}) = \expect{}{f(Y)}
\]
for some nonlinear function $f$ (in our case, the ``$k^{th}$ root'' function 
applied to a polynomial).
This is of course not true in general, but (as we shall see in 
Section~\ref{sec:laguerre}) seems to occur more generally in this context.

\section{Singular Values}\label{sec:laguerre}

For the sake of observing the parallels with the previous section, we will 
quickly discuss the case of singular values.
Our main tool will be a theorem that first appeared (in spirit) in 
\cite{aurelien_ff} but was proved in the generality of $\SPM$--bi-invariance in 
\cite{lg}.

\begin{theorem}\label{thm:rect_conv}
Let $A$ and $B$ be $\mm \times \nn$ independent random matrices with $B$ being 
$\SPM$--bi--invariant.
Furthermore, let $P, Q$ be power series for 
which the operators $P(\partial_x \partial_y)$ and $Q(\partial_x \partial_y)$ 
satisfy
\[
y^{\nn-\mm} \mydet{x y I - AA^T}
= P(\partial_x \partial_y) \applyto{ x^\mm y^\nn }
\AND
y^{\nn-\mm} \mydet{x y I - BB^T}
= Q(\partial_x \partial_y) \applyto{ x^\mm y^\nn }.
\]
Then 
\begin{equation}\label{eq:rect_conv}
\expect{Q, R}{ y^{\nn-\mm} \mydet{x y I - (A + Q B R)(A + Q B R)^T}}
= 
(PQ)(\partial_x \partial_y) \applyto{ x^\mm y^\nn }
\end{equation}
where $PQ$ denotes the multiplication of the two power series.
\end{theorem}

As mentioned in Section~\ref{sec:prelims}, $X_{\nn, \mm}$ 
is $\SPM$-bi--invariant, and so Theorem~\ref{thm:rect_conv} gives us (in 
theory) a way to compute
\[
\expect{}{ \mydet{x I - (A + \sqrt{\theta} X_{\mm, \nn})(A + \sqrt{\theta} 
X_{\mm, 
\nn})^\trans }}
\]
for any $A \in \MM_{\mm, \nn}$ simply by knowing 
\[
\expect{}{ \mydet{x I - AA^\trans}}
\AND
\expect{}{ \mydet{x I - \theta X_{\mm, \nn}X_{\mm, \nn}^\trans}}.
\]
This can again be derived from a result of Edelman which showed that the 
expected 
characteristic polynomial of the Wishart Ensemble is a Laguerre 
polynomial \cite{edelman_thesis}.
In Section~\ref{sec:cp} we will give an alternative derivation of the 
following fact (see Lemma~\ref{lem:myCLT}):
\begin{lemma}\label{lem:laguerre}
For all integers $\mm, \nn > 0$, we have 
\[
\expect{}{ y^{\nn-\mm} \mydet{x y I - \theta X_{\mm, \nn}X_{\mm, \nn}^\trans}}
=
e^{-\theta \partial_x \partial_y} \applyto{ x^\mm y^\nn }
\]
where we are writing $e^{-\theta \partial_x \partial_y}$ as shorthand for the 
operator 
\[
e^{-\theta \partial_x \partial_y} = \sum_{i=0}^\infty \frac{(-\theta)^i}{i!} 
\partial_x^{i} \partial_y^i.
\]

\end{lemma}

A combination of Theorem~\ref{thm:rect_conv} and Lemma~\ref{lem:laguerre} 
imply that for a given matrix $B \in \MM_{\mm, \nn}$ with 
\[
p_{BB^\trans}(x) = \mydet{x I - BB^\trans}
\]
we have the formula
\begin{equation}\label{eq:laguerre}
y^{\nn-\mm} \expect{}{\mydet{x y I - (B + X_{\mm, \nn})(B + X_{\mm, 
\nn})^\trans}}
= 
e^{-\theta \partial^2} \applyto{ y^{\nn-\mm} p_{BB^\trans}(x y) }.
\end{equation}

We now proceed similarly to Section~\ref{sec:hermite}.
For the sake of simplicity, we will assume $\nn \geq \mm$ and then find a  
differential equation that \eqref{eq:laguerre} satisfies.
Let $p(x, y)$ be an arbitrary polynomial and let
\[
q(x, y, t) = e^{-\theta \partial_x \partial_y} \applyto{ p(x, y) }.
\]
Now define the root path $r(\theta)$ so that $q(r(\theta), 1, \theta) = 0$.
Taking a derivative in $\theta$, we have
\[
0 = \frac{\partial}{\partial \theta} q(r(\theta), 1, \theta) 
= q_1(r(\theta), 1, \theta) r'(\theta) + q_3(r(\theta), 1, \theta)
\]
so (setting $\theta = 0$)
\[
0 = p_1(r, 1) r' - p_{12}(r, 1) 
\]
or
\begin{equation}\label{eq:LDE}
r'(0) = \frac{p_{12}(r, 1)}{p_1(r, 1)}.
\end{equation}
We can now use the fact that $x$ and $y$ play similar roles in 
\eqref{eq:laguerre}, 
so that for $p(x, y) = y^{\nn-\mm} p_{BB^\trans}(x y)$ we have
\[
p_1 = y^{\nn-\mm+1}p'_{BB^\trans}(x y)
\AND
p_2 = (\nn-\mm) y^{\nn-\mm-1}p_{BB^\trans}(x y) + x y^{\nn-\mm}p'_{BB^\trans}(x 
y).
\]
Hence we can write
\[
p_2 = (\nn-\mm) \frac{p}{y} + \frac{x}{y} p_1
\AND
p_{12} = (\nn-\mm+1) \frac{p_1}{y} + \frac{x}{y} p_{11}
\]
Now if we set $y = 1$, we see that \eqref{eq:LDE} becomes 
\[
r'(0) = (\nn-\mm+1) + r \frac{p_{11}(r, 1)}{p_1(r, 1)}.
\]
and so for $x = \lambda_i$, \eqref{eq:deriv_trick} implies the evolution of the 
roots of $p_{BB^\trans}$ satisfy the differential equations
\[
\frac{\partial \lambda_i}{\partial \theta}
= (\nn-\mm+1) + \lambda_i \sum_{j \neq i} \frac{2}{\lambda_i - \lambda_j}
= \nn + \sum_{j \neq i} \frac{\lambda_i + \lambda_j}{\lambda_i - 
\lambda_j}
\]
So, in particular, we see that the evolution again exhibits behavior similar to 
the drift term in the Brownian 
motion system \eqref{eq:laguerre_SDE}.
Furthermore, one can consider this correspondence to be ``curious'' for the 
same reasons as the Hermite case.

\section{A tool for generalized singular values}\label{sec:poly}

%

We start by briefly reviewing the generalized singular value decomposition and 
and its associated ``characteristic polynomial'' as motivation.

\subsection{The generalized singular value decomposition}\label{sec:gsvd}

The {\em generalized singular value decomposition (GSVD)} is often described as 
a decomposition on a pair of matrices, however we prefer the approach of 
\cite{vanloan} of using a single block matrix.
For simplicity, we will restrict to the case when this block matrix has full 
column rank, however we will find a way to eliminate this requirement shortly.  

For a matrix $M \in \MM_{\ss+\tt, \kk}$ with rank $\kk$ and block form
\[
M = \begin{bmatrix}
M_1 \\
M_2
\end{bmatrix}
\begin{array}{c} \ss \\ \tt \end{array}
\]
the GSVD provides a simultaneous decomposition of $M_1$ and $M_2$ into 
\[
M_1 = U_1 C H
\AND
M_2 = U_2 S H
\]
where 
\begin{itemize}
\item $U_1 \in \MM_{\ss, \ss}$ satisfies $U_1^\trans U_1 = U_1 U_1^\trans
= I_{\ss}$
\item $U_2 \in \MM_{\tt, \tt}$ satisfies $U_2^\trans U_2 = U_2 U_2^\trans 
= I_{\tt}$
\item $C \in \MM_{\ss, \kk}$ and $S \in \MM_{\tt, \kk}$ are 
pseudo-diagonal with $C^\trans C + S^\trans S = I_{\kk}$, and
\item $H$ is an $\kk \times \kk$ invertible matrix.
\end{itemize}
The diagonal entries of $C$ and $S$ satisfy $c_i^2 + s_i^2 = 1$, and 
because of this, these matrices are often referred to as {\em cosine} and 
{\em sine} matrices.

The standard singular value decomposition can then be recovered by letting $M_2 
= I$ and noting that this implies that $M_1$ has the decomposition
\[
M_1 = U_1 C S^{-1} U_2.
\]
When $M$ has rank $\kk$, there is an easy way to calculate the generalized 
singular values without needing to form the entire decomposition.
Letting $W_1 = M_1^\trans M_1$ and $W_2 = M_2^\trans M_2$, it is easy to check 
that 
\begin{equation}\label{eq:W}
W = (W_1 + W_2)^{-1/2} W_1 (W_1 + W_2)^{-1/2}
\end{equation}
is a positive semidefinite Hermitian matrix that has the same 
eigenvalues\footnote{What one actually calls the ``generalized singular 
values'' ($c_i$?, $c_i/s_i$?) is subject to interpretation, but all 
interpretations are functions of the eigenvalues of $C^\trans C$, which will  
always be our subject of interest in this paper.} as $C^\trans C$.
Letting
\begin{equation}\label{eq:pregsvdcp}
\mydet{x I - (W_1 + W_2)^{-1/2} W_1 (W_1 + W_2)^{-1/2}} 
= \mydet{W_1 + W_2}^{-1} 
\mydet{ x (W_1 + W_2) - W_1 }.
\end{equation}
and so we can use the characteristic polynomial
\begin{equation}\label{eq:gsvdcp}
\mydet{ x (W_1 + W_2) - W_1 }
= 
\mydet{ (x-1) W_1 + x W_2 }
\end{equation}
which has the added benefit that it does not require $W_1 + W_2$ to be 
invertible.
\subsection{A generalized characteristic polynomial}\label{sec:cp}

For matrices $A \in \MM_{\ss, \kk}$ and $B \in \MM_{\tt, \kk}$ we will consider 
the polynomial 
\begin{equation}\label{eq:xyzcp}
p_{A, B}(x, y, z) = \mydet{x I + y A^\trans A + z B^\trans B}.
\end{equation}
Note that $p_{A, B}(x, 0, -1)$ reduces to the polynomial $p_{BB^\trans}(x)$ 
from Section~\ref{sec:laguerre} and $p_{A, B}(0, 
z-1, 
z)$ reduces to (\ref{eq:gsvdcp}). 
Hence understanding the evolution of the polynomial
\begin{equation}\label{eq:cpX}
p_{A + \sqrt{\theta} X_{\ss, \kk}, B + \sqrt{\theta} X_{\tt, \kk}}
\end{equation}
will allow us to derive the information we want regarding both decompositions.

The main property of $X_{\mm, \nn}$ that we will exploit in our analysis is the 
fact that it is $\SPM$--bi-invariant. 
In particular, it was shown in \cite{lg} that for any $\SPM$--bi-invariant 
matrices $C, D$, the coefficients of
\[
\expect{}{ p_{A + C, B + D}(x, y, z) }
\]
are each multilinear functions of the coefficients of 
\[
\expect{}{ p_{A, B}(x, y, z) }
\AND
\expect{}{ p_{C, D}(x, y, z) }.
\]

To state the correspondence, it will be beneficial (for the moment) to instead 
work with a slight transformation of the polynomials $p_{A, B}$:
\[
q_{A, B}(x, y, z) = y^{\ss}z^{\tt} p(x, 1/y, 1/z).
\]
The results of \cite{lg} then imply the following:

\begin{theorem}\label{thm:lg}
Let $A, C \in \MM_{\ss, \kk}$ and $B, D \in \MM_{\tt, \kk}$ be random matrices 
with $C$ and $D$ $\SPM$--bi-invariate.
Now let $F$ and $G$ be bivariate polynomials for which 
\[
\expect{}{q_{A, B}(x, y, z) }
= 
F(\partial_x \partial_y, \partial_x \partial_z)
\applyto{ x^{\kk} y^{\ss} z^{\tt}}
\AND
 \expect{}{q_{C, D}(x, y, z) }
=
G(\partial_x \partial_y, \partial_x \partial_z) 
\applyto{ x^{\kk} y^{\ss} z^{\tt}}.
\]
Then 
\[
\expect{}{q_{A + C, B + D}(x, y, 
z)} 
= 
F(\partial_x \partial_y, \partial_x \partial_z)
G(\partial_x \partial_y, 
\partial_x \partial_z)
\applyto{ x^{\kk} y^{\ss} z^{\tt}}.
\]
\end{theorem}

Given Theorem~\ref{thm:lg}, an obvious next step to computing the polynomial in 
(\ref{eq:cpX}) would be to find the value of the polynomial 
$\expect{}{q_{a X_{\ss, \kk}, b 
X_{\tt, 
\kk}}(x, y, z)}$ for general $a, b \in \R$.
This can be done combinatorially (as in Theorem~\ref{lem:hermite}), but we 
will find it easier (and more instructive) to use a central limit theorem 
argument.

\begin{lemma} \label{lem:myCLT}
Let $\{ A_i \}_{i=1}^\infty \subseteq \MM_{\ss, \kk}$ and $\{ B_i 
\}_{i=1}^\infty \subseteq \MM_{\tt, \kk}$ be sequences of independent random 
$\SPM$--bi-invariant matrices for which 
\[
\expect{}{\mytr{A_i A_i^\trans}} 
= \theta_1 \ss \kk \AND 
\expect{}{\mytr{B_i B_i^\trans}} 
= \theta_2 \tt \kk
\]
for all $i$.
Define the random matrices
\[
C_{\mm} 
=  \frac{\sum_{i=1}^n A_i}{\sqrt{\mm}}
\AND 
D_{\mm} 
=  \frac{\sum_{i=1}^n B_i}{\sqrt{\mm}}.
\]
Then
\begin{equation}\label{eq:bmpoly}
\lim_{\mm \to \infty} 
\expect{}{q_{C_{\mm}, D_{\mm}}(x, y, z)}
= 
e^{\theta_1 \partial x \partial y + \theta_2 \partial_x \partial_z} 
\applyto{x^{\kk} y^{\ss} z^{\tt}}
\end{equation}
\end{lemma}
\begin{proof}
Due to the boundedness of the underlying random variables, we can write
\begin{align*}
\expect{}{ q_{A_i/\sqrt{\mm}, B_i/\sqrt{\mm}} (x, y, z) }
&= x^{\kk} y^{\ss} z^{\tt} 
+ \frac{\theta_1 \ss \kk}{\mm} x^{\kk-1} y^{\ss-1} z^{\tt} 
+ \frac{\theta_2 \tt \kk}{\mm} x^{\kk-1} y^{\ss} z^{\tt-1} + 
O\left(\frac{1}{m^2} \right)
\\&=
1 + \frac{\theta_1}{\mm} \partial_x \partial_y + \frac{\theta_2}{\mm} 
\partial_x \partial_z + 
O\left(\frac{1}{m^2} \right)
\applyto{x^{\kk} y^{\ss} z^{\tt}}
\end{align*}
as $m \to \infty$.
So by Theorem~\ref{thm:lg}, we have
\[
\expect{}{ q_{C_{\mm}, D_{\mm}} (x, y, z) }
= 
\left(1 + \frac{\theta_1}{\mm} \partial_x \partial_y + \frac{\theta_2}{\mm} 
\partial_x \partial_z + O\left(\frac{1}{{\mm}^2}\right) \right)^{\mm} 
\applyto{x^{\kk} 
y^{\ss} z^{\tt}}
\]
which converges to the claimed polynomial as $\mm \to \infty$.
 \end{proof}

We now show that (\ref{eq:bmpoly}) is (in fact) the same polynomial as 
$\expect{}{q_{a X_{\ss, \kk}, b X_{\tt, 
\kk}}(x, y, z)}$.

\begin{corollary}\label{cor:formula}
\[
\expect{}{q_{a X_{\ss, \kk}, b X_{\tt, 
\kk}}(x, y, z)}
= 
e^{a^2 \partial x \partial y + b^2 \partial_x \partial_z} 
\applyto{x^{\kk} y^{\ss} z^{\tt}}
\]
\end{corollary}
\begin{proof}
Let $G_1 \in \MM_{\ss, \kk}$ and $G_2 \in \MM_{\tt, \kk}$ be matrices of 
independent standard Gaussians and let $a$ and $b$ be real numbers.
We will apply Lemma~\ref{lem:myCLT} in the case where the $A_i$ are independent 
copies of $a G_1$ and the $B_i$ are independent copies of $b G_2$.

On the one hand, it is easy to see using the properties of standard Gaussians 
that 
\[
C_{\mm} 
= \frac{\sum_{i=1}^n A_i}{\sqrt{\mm}}
= a G_1
\AND 
D_{\mm} 
= \frac{\sum_{i=1}^n B_i}{\sqrt{\mm}}
= b G_2
\]
for all $\mm$.
Hence we have
\[
\lim_{\mm \to \infty}  \expect{}{q_{C_{\mm}, D_{\mm}}(x, y, z)}
= \expect{}{ q_{a G_1, b G_2}(x,y,z)} .
\]

On the other hand, it is easy to calculate that 
\[
\expect{}{\mytr{a^2 G_1^\trans G_1}} = a^2 \ss \kk
\AND
\expect{}{\mytr{b^2 G_2^\trans G_2}} = b^2 \tt \kk
\]
so Lemma~\ref{lem:myCLT} implies that 
\[
\lim_{\mm \to \infty} \expect{}{q_{C_{\mm}, D_{\mm}}(x, y, z)}
= e^{\theta_1 \partial x \partial y + \theta_2 \partial_x \partial_z} 
\applyto{x^{\kk} y^{\ss} z^{\tt}}
\]
Equating the two gives us 
\[
\expect{}{ q_{a G_1, b G_2}(x,y,z)} 
= 
e^{a^2 \partial x \partial y + b^2 \partial_x \partial_z} 
\applyto{x^{\kk} y^{\ss} z^{\tt}}
\]

The remainder of the proof follows from the fact that $q_{A, B}$ is at most 
quadratic in the entries of the matrices $A$ and $B$.
So when those entries are random variables, the only contributions to the 
polynomial come from the first two moments.
Hence having a matrix of independent standard Gaussians will give the same 
result as any set of independent mean $0$, variance $1$ random variables, and 
so we have
\[
\expect{}{q_{a X_{\ss, \kk}, b X_{\tt, 
\kk}}(x, y, z)}
= 
\expect{}{q_{a G_1, b G_2}(x, y, z)}
= 
e^{a^2 \partial x \partial y + b^2 \partial_x \partial_z} 
\applyto{x^{\kk} y^{\ss} z^{\tt}}
\]
as claimed.
\end{proof}

Theorem~\ref{thm:lg} combined with Corollary~\ref{cor:formula} gives us the 
following formula for any matrices $A \in \MM_{\ss, \kk}$ and $B \in \MM_{\tt, 
\kk}$:
\begin{equation}\label{eq:qs}
\expect{}{q_{A + a X_{\ss, \kk}, B + b X_{\tt, \kk}}(x, y, z }
= 
e^{a^2 \partial x \partial y + b^2 \partial_x \partial_z} \applyto { q_{A, 
B}(x, y, z) }.
\end{equation}
We end this subsection by showing how one can translate this back to the 
polynomials we are truly interested in.
For the sake of notation, for a polynomial $f = f(y)$, let us define the 
operator $U^y_k$ as
\begin{equation}\label{eq:uoper}
U^y_k \applyto{p(y)} = y^k p \left(\frac{1}{y} \right).
\end{equation}

\begin{lemma}\label{lem:trans}
Let $i, j$ be integers with $j \leq s$.
Then
\[
U_y^s \circ e^{\theta \partial_x \partial_y} \circ U_y^s \applyto{x^i y^j} 
= (1 + \theta y \partial_x)^{s-j} \applyto{ x^i y^j}
\]
\end{lemma}
\begin{proof}
We compute: 
\begin{align*}
U_y^s \circ e^{\theta \partial_x \partial_y} \circ U_y^s \applyto{x^i y^j}
&= U_y^s \circ e^{\theta \partial_x \partial_y} \applyto{x^i y^{s-j}}
\\&= U_y^s\applyto{\sum_k \frac{\theta^k}{k!} {\partial_x^k \partial_y^k}  
\applyto{ x^i y^{s-j}}}
\\&= U_y^s \applyto{ \sum_k \theta^k \partial_x^k \binom{s-j}{k} x^i y^{s-j-k}}
\\&= (1 + \theta y \partial_x)^{s-j} \applyto{ x^i y^j}.
\end{align*}
\end{proof}

\subsection{Returning to the decompositions}

We now return the respective decompositions to give a proof of 
Lemma~\ref{lem:laguerre} and the corresponding result for generalized singular 
values.
Consider formula (\ref{eq:qs}) in the case when $A$ and $B$ are $0$ matrices 
and $a = b = 1$.
Then $q_{A, B}(x, y, z) = x^\kk$ and so Lemma~\ref{lem:trans} implies
\[
\expect{}{ p_{X_{\ss, \kk}, X_{\tt, \kk}}(x, y, z) } 
= (1 + y \partial_x)^{\ss}(1 + z \partial_x)^{\tt} \applyto{ x^{\kk}}.
\]
As noted at the beginning of Section~\ref{sec:cp}, we can get the singular 
value 
polynomial by setting $y = -1$ and $z = 0$:
\[
\expect{}{ \mydet{x I - X^\trans_{\ss, \kk}X_{\ss, \kk} }} 
= (1 - \partial_x)^{\ss} \applyto{ x^{\kk}}
= (-1)^\kk \kk! L^{(\ss-\kk)}_\kk(x)
\]
where $L_n^{(\alpha)}(x)$ is a Laguerre polynomial \cite{szego}.
Setting $x = 0$ and $y = z-1$, on the other hand, gives us the generalized 
singular value polynomial:
\[
\expect{}{ \mydet{(z-1) X^\trans_{\ss, \kk}X_{\ss, \kk} + z X^\trans_{\tt, 
\kk}X_{\tt, \kk} }}
=  (1 + (z-1)\partial_x)^{\ss}(1 + z\partial_x)^{\tt} \applyto{ x^{\kk}}
= P^{(\tt-\kk, \ss-\kk)}_\kk(2x-1)
\]
where $P_n^{(\alpha, \beta)}(x)$ is a Jacobi polynomial \cite{szego}.
For those familiar with random matrix theory, these two results are perhaps not 
surprising as the random matrix ensembles one would get by using i.i.d. 
standard normals random variables in the random matrices $X_{\kk, \kk}, X_{\ss, 
\kk}$ and $X_{\tt, \kk}$ are often called the Hermite, 
Laguerre, and Jacobi ensembles (for this reason) \cite{forrester}.

\section{Brownian Motion}\label{sec:bm}

\renewcommand{\dd}{k}
\newcommand{\pp}{\tilde{p}}
\newcommand{\xx}{\tilde{x}}
\newcommand{\yy}{\tilde{y}}
\newcommand{\zz}{\tilde{z}}
\newcommand{\ww}{\tilde{w}}

Using the results in Section~\ref{sec:cp}, we have reduced the problem of 
interest to understanding the effects of a certain differential operator on a 
generalized characteristic polynomial.
The operator of interest is $Q^\theta_{s, t}$ defined by 
\begin{equation}\label{eq:bm}
Q^\theta_{s, t}[x^i y^j z^k] \mapsto (1 + \theta y \partial_x)^{s-j} (1 + 
\theta z \partial_x)^{t-k} \applyto{x^i y^j z^k}
\end{equation}
and then extended to general polynomials linearity.

To understand the evolution of these polynomials, we first derive a 
differential equation that this operator satisfies.
For a fixed polynomial $p(x, y, z)$, we define the polynomial 
\[
\pp(x, y, z, \theta) 
= Q^{\theta}_{s, t} \applyto{ p(x, y, z) }
\]
and then consider variables $\xx = \xx(\theta), \yy = \yy(\theta), \zz = 
\zz(\theta)$ for which 
\[
\pp(\xx(\theta), \yy(\theta), \zz(\theta), \theta) = 0.
\]
In particular, we can take a derivative in $\theta$ to get 
\[
\pp_1 \xx' + \pp_2 \yy' + \pp_3 \zz' + \pp_4 = 0.
\]
\begin{lemma}
\[
\pp_4(x, y, z, 0) = (\ss y + \tt z)\pp_1(x, y, z, 0) - y^2 \pp_{12}(x, y, z, 0) 
- 
z^2 
\pp_{13}(x, y, z, 0)
\]
\end{lemma}
\begin{proof}
Write 
\[
p(x, y, z) = \sum_{i, j, k} c_{i, j, k} x^i y^j z^k
\]
where the degree of $y$ is at most $s$ and the degree of $z$ is at most $t$.
Then 
\begin{align*}
Q^{\theta}_{s, t} \applyto{ p(x, y, z) }
&= \sum_{i, j, k} c_{i, j, k} (1 + \theta y \partial_x)^{\ss-j} (1 + \theta z 
\partial_x)^{\tt-k} \applyto{x^i y^j z^k}
\\&= \sum_{i, j, k} c_{i, j, k} \left(1 + \theta( (\ss-j) y \partial_x + 
(\tt-k) z 
\partial_x) + O(\theta^2) \right) \applyto{x^i y^j z^k}
\\&= \sum_{i, j, k} c_{i, j, k} 
\left( x^i y^j z^k 
+ i(\ss-j) \theta x^{i-1}y^{j+1}z^k 
+ i(\tt-k) \theta x^{i-1}y^{j}z^{k+1}
\right) + O(\theta)^2
\end{align*} 
and so 
\[
\frac{\partial}{\partial \theta} Q^{\theta}_{\ss, \tt} \applyto{ p(x, y, z) } 
\bigg|_{\theta = 0}
= \sum_{i, j, k} i c_{i, j, k} 
\left( 
(\ss-j) x^{i-1}y^{i+1}z^k 
+(\tt-k) x^{i-1}y^{j}z^{k+1}
\right).
\]
Now we rewrite
\[
i (\ss-j) x^{i-1}y^{j+1}z^k 
= i \ss x^{i-1}y^{j+1}z^k - i j x^{i-1}y^{j+1}z^k
= \ss y \partial_x \applyto{x^i y^j z^k} - y^2 \partial x \partial y 
\applyto{x^i 
y^j z^k}
\]
and similarly 
\[
i (\tt-k) x^{i-1}y^{j}z^{k+1} 
= i \tt x^{i-1}y^{j}z^{k+1} - i k x^{i-1}y^{j}z^{k+1}
= \tt z \partial_x \applyto{x^i y^j z^k} - z^2 \partial x \partial z 
\applyto{x^i 
y^j z^k}.
\]
Hence 
\[
\pp_4(x, y, z, 0) 
= \frac{\partial}{\partial \theta} Q^{\theta}_{\ss, \tt} \applyto{ p(x, y, z) } 
\bigg|_{\theta = 0}
= (\ss y + \tt z) p_1(x, y, z) - y^2 p_{12}(x, y, z) - z^2 p_{13}(x, y, z)
\]
and the result follows from the fact that $\pp(x, y, z, 0) = p(x, y, z)$.
\end{proof}

It is worth mentioning that the equation we have derived: 
\[
\pp_1 \xx' + \pp_2 \yy' + \pp_3 \zz' 
= 
- (\ss y + \tt z)\pp_1 
+ y^2 \pp_{12} 
+ z^2 \pp_{13}
\]
still has a bit of freedom to be manipulated due to the homogeneity if $p$.
How we do that manipulation will depend on how we intend to use the polynomial.
For example, if our goal is going to be to plug in $\yy = \ww - 1$ and $\zz = 
\ww$ at 
some point, then the terms 
\[
y^2 p_{12}
\AND
z^2 p_{13}
\]
will not reduce to functions of $\ww$ easily.
We can fix this as follows:
\begin{lemma}
Let $p(x, y, z)$ be a $\dd$-homogeneous polynomial.
Then 
\[
y^2 p_{12} + z^2 p_{13} = (\dd-1)(y + z)p_1 - x(y+z)p_{11} - y z(p_{12} + 
p_{13}).
\]  
\end{lemma}
\begin{proof}
Since $p$ is $\dd$ homogeneous, $p_1$ is $\dd-1$ homogeneous, and so 
\[
x f_{11} + y f_{12} + z f_{13} = (\dd-1) f_1
\]
Solving for $f_{13}$ (and separately for $f_{12}$) then gives
\[
z^2 f_{13} = -x z f_{11} - y z f_{12} + z (d-1)f_{1}
\AND
y^2 f_{12} = -x y f_{11} - y z f_{13} + y (d-1)f_{1}.
\]
and the lemma follows by adding them together.
\end{proof}
Hence we have derived the following equation for $p$ at the point $(\xx, \yy, 
\zz)$ for $\theta = 0$:
\begin{equation}\label{eq:jacobipde}
p_1 \xx' + \pp_2 \yy' + \pp_3 \zz' 
= 
-(\ss-\dd+1)\yy p_1 
-(\tt-\dd+1) \zz p_1 
-\xx(\yy+\zz)p_{11} 
-\yy \zz(p_{12} + p_{13}).
\end{equation}

%
%

\subsection{Jacobi Process}\label{sec:gsvdbm}

We now make the substitution described in \eqref{eq:gsvdcp}, letting
\[
h(x, u) = p(x, u-1, u)
\]
and setting $\yy = \ww-1$ and $\zz = \ww$ in (\ref{eq:jacobipde}) we get 
\begin{align*}
\xx' h_1(\xx, \ww) + \ww' h_2(\xx, \ww) 
=
&-(\ss-\kk+1)(\ww-1) h_1(\xx, \ww) - (\tt - \kk + 1) \ww h_1 (\xx, \ww)
\\&-\xx(2\ww - 1)h_{11}(\xx, \ww) - \ww(\ww-1)h_{12}(\xx, \ww)
\end{align*}
which for $\xx = 0$ simplifies to 
\begin{equation}\label{eq:prefinal}
\ww' \frac{h_2}{h_1}
=
-(\ss-\kk+1)(\ww-1) - (\tt - \kk + 1) \ww - \ww(\ww-1)\frac{h_{12}}{h_1}.
\end{equation}

We would like to compare this to the evolution of the random matrix version of 
the Jacobi process given in \eqref{eq:jacobi_SDE} with parameters $p = n_1$ and 
$q = n_2$.
At first glance, these are not the same, but we claim that this is because 
\eqref{eq:prefinal} is actually significantly more general than  
\eqref{eq:jacobi_SDE}.

Before addressing this, however, we wish to point out that not only are 
(\ref{eq:prefinal}) and 
(\ref{eq:jacobi_SDE}) different, they are {\em fundamentally} different.
In particular, (\ref{eq:jacobi_SDE}) is an eigenvalue evolution that only 
depends on the original eigenvalues.
That is, if we let 
\[
A_\theta 
= A + \sqrt{\theta} X_{\ss, \kk}
\AND
B_\theta 
= B + \sqrt{\theta} X_{\tt, \kk}
\]
then the evolution of (\ref{eq:jacobi_SDE}) would imply that one can compute 
the eigenvalues of
\begin{equation}\label{eq:matrixtheta}
(A_\theta A_\theta^T  + B_\theta B_\theta^T)^{-1/2} A_\theta A_\theta^T 
(A_\theta A_\theta^T  + B_\theta B_\theta^T)^{-1/2}
\end{equation}
as a function of the eigenvalues of 
\[
(A_0 A_0^T  + B_0 B_0^T)^{-1/2} A_0 A_0^T (A_0 A_0  + B_0 B_0)^{-1/2}.
\]
This is not the case for (\ref{eq:prefinal}), and for good reason. 
As we will see in Section~\ref{sec:simulation}, the eigenvalues of 
(\ref{eq:matrixtheta}) 
will depend on (among other things) the relationship between the eigenvectors 
of $A_0A_0^T$ and $B_0B_0^T$, and so to some extent it is actually surprising 
that we can get all of the necessary information even from the complete 
polynomial 
$h(\xx, \ww)$. 

\subsection{The matrix Jacobi process, revisited}\label{sec:gsvdbm2}

To see why the two solution differ, let us first review how the standard matrix 
Jacobi process is defined.
As we saw in Section~\ref{sec:gsvd}, the method for obtaining the generalized 
singular values of
\[
M = \begin{bmatrix}
M_1 \\
M_2
\end{bmatrix}
\begin{array}{c} \ss \\ \tt \end{array}
\]
was via the formula (\ref{eq:W}) where $W_1 = M_1^\trans M_1$ and $W_2 = 
M_2^\trans M_2$.
Note that $W_1 + W_2 = M^\trans M$, and that 
\[
M_1 = 
\begin{bmatrix}
I_{\ss} & 0
\end{bmatrix}
M
\]
and so another way to write (\ref{eq:W}) would be as 
\[
W = (M^\trans M)^{-1/2} M^\trans 
\begin{bmatrix}
I_{\ss} \\
0
\end{bmatrix}
\begin{bmatrix}
I_{\ss} & 0
\end{bmatrix}
M 
(M^\trans M)^{-1/2} M^\trans.
\]
One can simplify this using the (normal) singular value decomposition: letting 
$M = U \Sigma V^\trans$, and simplifying, we get that 
\[
W = 
\begin{bmatrix}
I_{\kk} & 0 \\
\end{bmatrix}
U^\trans 
\begin{bmatrix}
I_{\ss} \\
0
\end{bmatrix}
\begin{bmatrix}
I_{\ss} & 0
\end{bmatrix}
U
\begin{bmatrix}
I_{\kk} \\
0 
\end{bmatrix}
= 
C^\trans C
\]
where $C$ is the upper $\ss \times \kk$ corner of the unitary matrix $U \in 
\MM_{\ss + \tt, \ss+\tt}$.

When $M$ is a matrix of Gaussians, the derived distribution on $U$ is the 
Haar-distribution.
This shows the equivalence of two methods for creating a Jacobi ensemble (a 
fact that has been used numerous times in the literature 
\cite{collins,edelmanCS}):
\begin{enumerate}
\item as the generalized singular values of a random Gaussian matrix
\item as the singular values of the upper corner of a Haar-distributed 
unitary matrix
\end{enumerate}
The issue occurs when one tries to turn this into a process.
The standard matrix Jacobi process uses the ``upper corner'' model, replacing 
the Haar distribution with a {\em unitary} Brownian motion (starting, possibly 
at a fixed unitary matrix) \cite{doumerc}.
And this is precisely where the two processes diverge --- for general matrices 
$A_\theta$, approximating the left eigenvectors of $A_\theta$ with a unitary 
Brownian motion centered at the left eigenvectors of $A_0$ is only that (an 
approximation).
The actual distribution will depend on the singular values of $A$, which (in 
turn) will depend on the relationship between the two 
blocks.

On the other hand, if one was able to artificially force the singular values of 
$A_\theta$ to remain the same throughout the process, then this extra effect 
(while not necessarily zero) will not change, and so one can then hope to 
derive an evolution equation which holds for all $A_\theta$.
And this is (essentially) what is being done in the ``upper corner process'' 
--- by using a unitary Brownian motion, one is forcing all of the singular 
values of 
the $A_\theta$ to be $1$ (for all $\theta$)\footnote{This is often how Brownian 
motion is {\em defined} --- as the process on the space of unitary matrices 
which has instantaneous change equal to that of the additive process 
\cite{mehta}.}.
And it is not hard to show that for any $A_\theta$ where this is the case, 
dynamics in (\ref{eq:prefinal}) at that $\theta$ can be computed explicitly in 
terms of the current point configuration:

\begin{lemma}
Let $A \in \MM_{(\ss+\tt), \kk}$ be such that $A^\trans A = I_{\kk}$.
Then $h_1(0, u) = h_2(0, u)$ and 
\[
\frac{h_{12}(0, \lambda_i)}{h_1(0, \lambda_i)} = 
\sum_{j \neq i} \frac{2 
}{\lambda_i - \lambda_j}
\]
where $\{ \lambda_i \}_{i=1}^\kk$ are the roots of $h(0, u)$.
\end{lemma}
\begin{proof}
Letting
\[
A = \begin{bmatrix}
A_1 \\
A_2
\end{bmatrix}
\begin{array}{c} \ss \\ \tt \end{array}
\]
we have that $A^\trans A = I_\kk$ is equivalent to $A_1^\trans A_1 + A_2^\trans 
A_2 = I$, so plugging these into the generalized characteristic polynomial 
(\ref{eq:xyzcp}), you get
\[
h(x, u) 
= p(x, u-1, u) 
= \mydet{x I + (u-1) A_1^\trans A_1^\trans + u (I - A_1^\trans 
A_1)}
= \mydet{(x + u) I - A_1^\trans A_1}.
\]
Hence derivatives in $x$ and derivatives in $u$ are the same, so (in 
particular) $h_1(0, u) = h_2(0, u)$.
For the same reason, we have
\[
\frac{h_{12}(0, u)}{h_1(0, u)} 
= \frac{h_{22}(0, u)}{h_2(0, u)}
\]
where if we let $g(u) = h(0, u)$, then we can again use (\ref{eq:deriv_trick}) 
to get
\[
\frac{h_{22}(0, \lambda_i)}{h_2(0, \lambda_i)}
= 
\frac{g''(\lambda_i)}{g'(\lambda_i)}
= 
\sum_{j \neq i} \frac{2}{\lambda_i - \lambda_j}
\]
as claimed.
\end{proof}

Hence whenever $A_\theta^\trans A_\theta = I_\kk$, (\ref{eq:prefinal}) reduces 
to
\begin{align*}
\frac{\partial \lambda_i}{\partial \theta} 
&= 
- (\ss-\kk + 1)(\lambda_i-1) 
- (\tt-\kk + 1) \lambda_i 
- \sum_{j \neq i} \frac{2 \lambda_i (\lambda_i - 1)}{\lambda_i - \lambda_j} 
\\&= 
- \ss(\lambda_i-1) 
- \tt \lambda_i 
+ \sum_{j \neq i} \frac{\lambda_i(1 - \lambda_j) +  \lambda_j (1 - 
\lambda_i)}{\lambda_i - \lambda_j} 
\end{align*}
matching the drift term in (\ref{eq:jacobi_SDE}) exactly.  
To end the section, we find it worth mentioning that the GSVD of a unitary 
matrix is a decomposition that is well-studied in its own right (known as the 
CS-decomposition) \cite{golub}.
This fact was noted in \cite{edelmanCS} as well, however only in the context of 
the Jacobi ensemble itself (not the process).
Hence this might be the more appropriate ``decomposition'' to associate 
directly to the matrix Jacobi process and, if this is the case, then it 
suggests that a random matrix process related to the true GSVD has yet to be 
considered in full generality.

\section{Simulation}\label{sec:simulation}

The purpose of this section is 3-fold.  
Firstly, we hope to convince the reader of the assertion made in 
Section~\ref{sec:gsvdbm} that the point process described by \textbf{4.} is 
fundamentally different than the point processes defined by \textbf{2.} and 
\textbf{3.} in the respect that future point configurations are not determined 
by the current point configuration alone.  

Secondly, we hope to explore the degree to which the $k$th root of the expected 
characteristic polynomial correspond to the expected value of the $k$th root.
The use of free probability as a tool for estimating random matrix statistics 
has a long history, hence there is hope that finite free probability has 
similar potential.
The obvious upside of finite free probability is that it holds in fixed 
dimensions, and so one might hope that in situations where a fixed dimensions 
ensemble has features that do not remain in large dimensions that finite free 
probability could give better results.
The obvious downside is that, unlike in the case of free probability, it is 
unclear how well the statistics that one is able to compute (the expected 
characteristic polynomials) relate to the statistics one might want to compute 
(moments or probabilities).
Thus we hope to convey some indication as to what the potential is in this 
respect.


Thirdly, we wish to explore the recurring pattern in this paper of the point 
process defined by the polynomial convolution matching the ``non-random'' or 
``as $\beta \to \infty$'' part of some random matrix process. 
While it is only a limited set, we wish to examine the $\beta = 1$ cases and 
$\beta = 2$ cases to see whether we can find support for such a claim 
experimentally and also to see whether there is some further intuition we can 
gain in this regard (like the appearance of monotonicity, for example).

\subsection{The experiment}

We first describe the experiment.
We set $\kk=4, \ss = 5, \tt = 10$, and use the following $\kk \times \kk$ 
diagonal matrices:\footnote{One would not be wrong to consider the choice of 
matrices $H$ and $K$ to be somewhat extreme and not representative of a typical 
application. 
One can show that even small differences between $H$ and $K$ result in 
different behaviors, but ``how different'' is unclear.
The choice of these particular matrices was made with the hopes of making 
various features visually recognizable.}
\[
D_1 =
\begin{bmatrix}
2 & & & \\
& 2 & & \\
& & 5 & \\
& & & 5
\end{bmatrix}
\quad
D_2 =
\begin{bmatrix}
10 & & & \\
& 10 & & \\
& & 10 & \\
& & & 10
\end{bmatrix}
\quad
H =
\begin{bmatrix}
1 & & & \\
& 1 & & \\
& & 1 & \\
& & & 500
\end{bmatrix}
\quad
K =
\begin{bmatrix}
500 & & & \\
& 1 & & \\
& & 1 & \\
& & & 1
\end{bmatrix}
\]
We add $(\ss-\kk)$ rows of $0$ to $D_1$ to form $F_1 \in \MM_{\ss, \kk}$ (and 
similarly to form $F_2 \in \MM_{\tt, \kk}$ from $D_2$).
We then pick constants $\delta, \epsilon$ and form
\[
A = \begin{bmatrix}
A_1 \\
A_2
\end{bmatrix}
= \delta
\begin{bmatrix}
F_1 H \\
F_2 H
\end{bmatrix}
\AND
B = \begin{bmatrix}
B_1 \\
B_2
\end{bmatrix}
=
\epsilon
\begin{bmatrix}
F_1 K \\
F_2 K
\end{bmatrix}.
\]
with $\delta, \epsilon$ chosen so that $\mytr{A^\trans A} = 
\mytr{B^\trans B} = 1$. 

Using (\ref{eq:W}), one can easily check that multiplying on the right by an 
invertible matrix does not change the eigenvalues of $C^\trans C$ and so the 
squares 
of the generalized singular values of $A$ and $B$ will be the same:

\[
\left( \frac{4}{13}, \frac{4}{13}, \frac{25}{34}, 
\frac{25}{34}\right)  \approx ( 0.308, 0.308, 0.735, 0.735 ).
\]

For each trial, we consider two independent point processes: one starting at 
$A^0 = A$ and the other at $B^0 = B$.  
At time step $i$, we generate a random matrix 
$Z_i \in \MM_{(\ss + \tt), \kk}$ with independent normal entries ($\mu = 0, 
\sigma^2 = 10^{-8}$) to add 
to both $A$ and $B$\footnote{The reason for adding the same matrix is to 
ensure that any differences that will appear are not a result of random effects.
However numerous tests were done using independent matrices and each gave 
similar results.}:
\[
A^i \leftarrow A^{i-1} + Z_i
\AND
B^i \leftarrow A^{i-1} + Z_i.
\]
We calculate the squares of the generalized singular values $c_{A^i}^2, 
c_{B^i}^2 \in \R^\kk$ 
(both in increasing order).
We then average over $500$ trials, obtaining the (average) squared generalized 
singular values $\overline{c}_{A^i}^2, 
\overline{c}_{B^i}^2 \in 
\R^\kk$.
These will be compared to the roots of the expected generalized singular value 
polynomials 
\[
\tilde{p}_{A^i}(x, y, z)
= 
Q^{(i \sigma^2)}_{\ss, \tt} \applyto{ p_A(x, y, z) }
\AND
\tilde{p}_{B^i}(x, y, z)
= 
Q^{(i \sigma^2)}_{\ss, \tt} \applyto{ p_B(x, y, z) }.
\]
To simplify the wording slightly, rather than writing (for example) ``notice 
that $\tilde{p}_{A^i}$ (respectively $\tilde{p}_{B^i}$) are'', we will write 
simply 
``notice that 
the $\tilde{p}_{M^i}$ are''
(essentially anywhere there is an $M$, it should be taken as being a statement 
for both $A$ and $B$, separately).

\subsection{Results}

%
%
%
%

We first hope to convey the fact that the evolution of $A$ and $B$ --- despite 
starting at the same generalized singular values and using the same ``random'' 
matrices --- are fundamentally different.
For this purpose we can examine the evolution of the point process over various 
time scales (Figure~\ref{fig:scales}).
There are, of course, similarities, it is easy to see that both $A$ and $B$  
exhibit repulsion between the paths early on in addition to converging to the 
roots of a Jacobi polynomial in the long term (both of which should be 
expected).
However $B$ seems to converge very quickly to the asymptotic limit 
whereas $A$ takes a (very) long time.
Of particular interest is the path of the largest value (in red); the path of 
$A$ 
actually moves {\em away} from the eventual asymptotic limit for some time 
during the middle range.
This is in direct contradiction to the other paths of $A$ and all of the 
paths of $B$ which seem to move in the direction of the asymptotic 
limit.

\begin{figure}[!h]
    \centering
    \begin{minipage}{.25\textwidth}
        \centering
        \includegraphics[width=1\linewidth, 
        height=0.15\textheight]{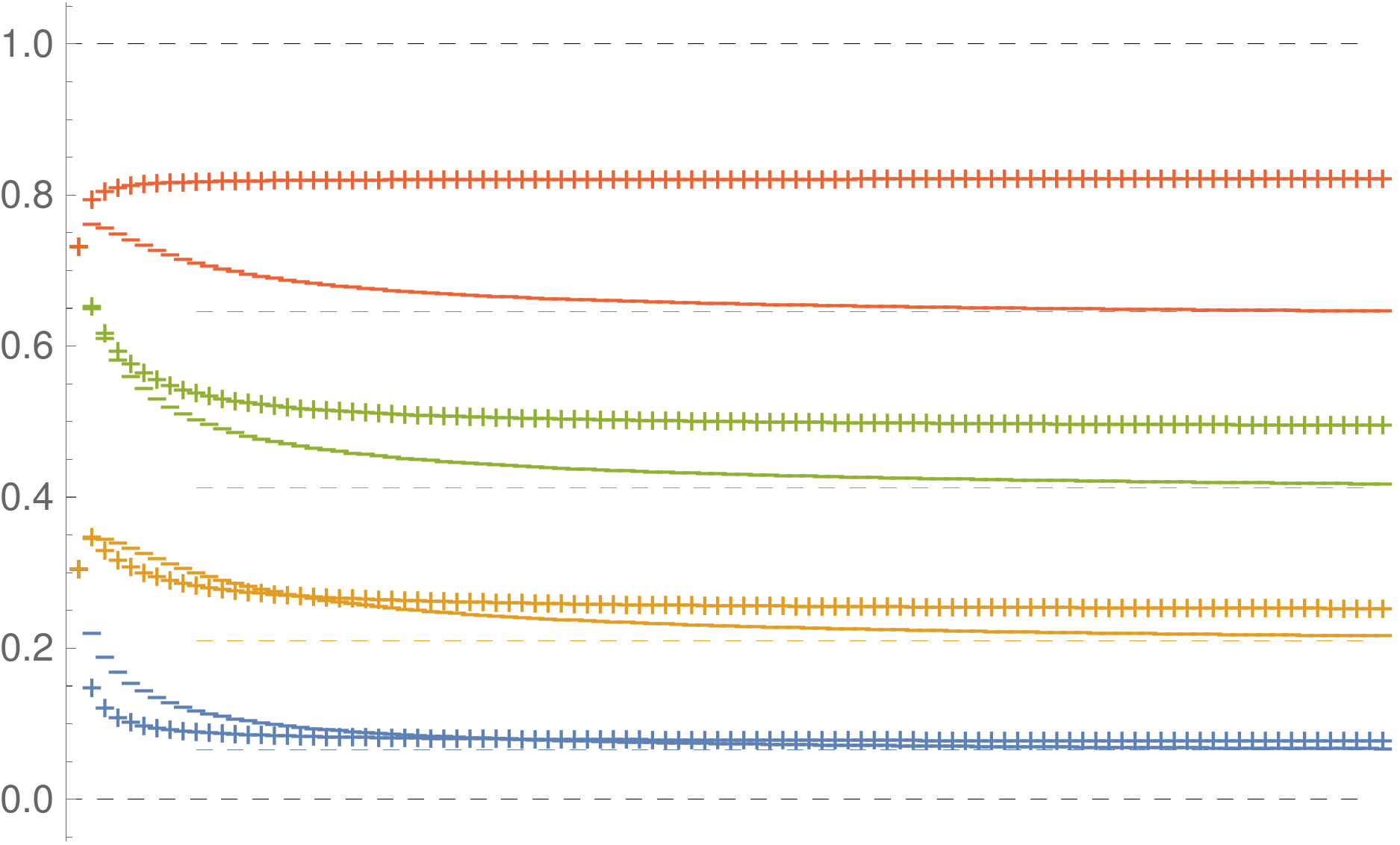}
First $10^3$ steps.
    \end{minipage}%
\qquad
    \begin{minipage}{0.25\textwidth}
        \centering
        \includegraphics[width=\linewidth, 
        height=0.15\textheight]{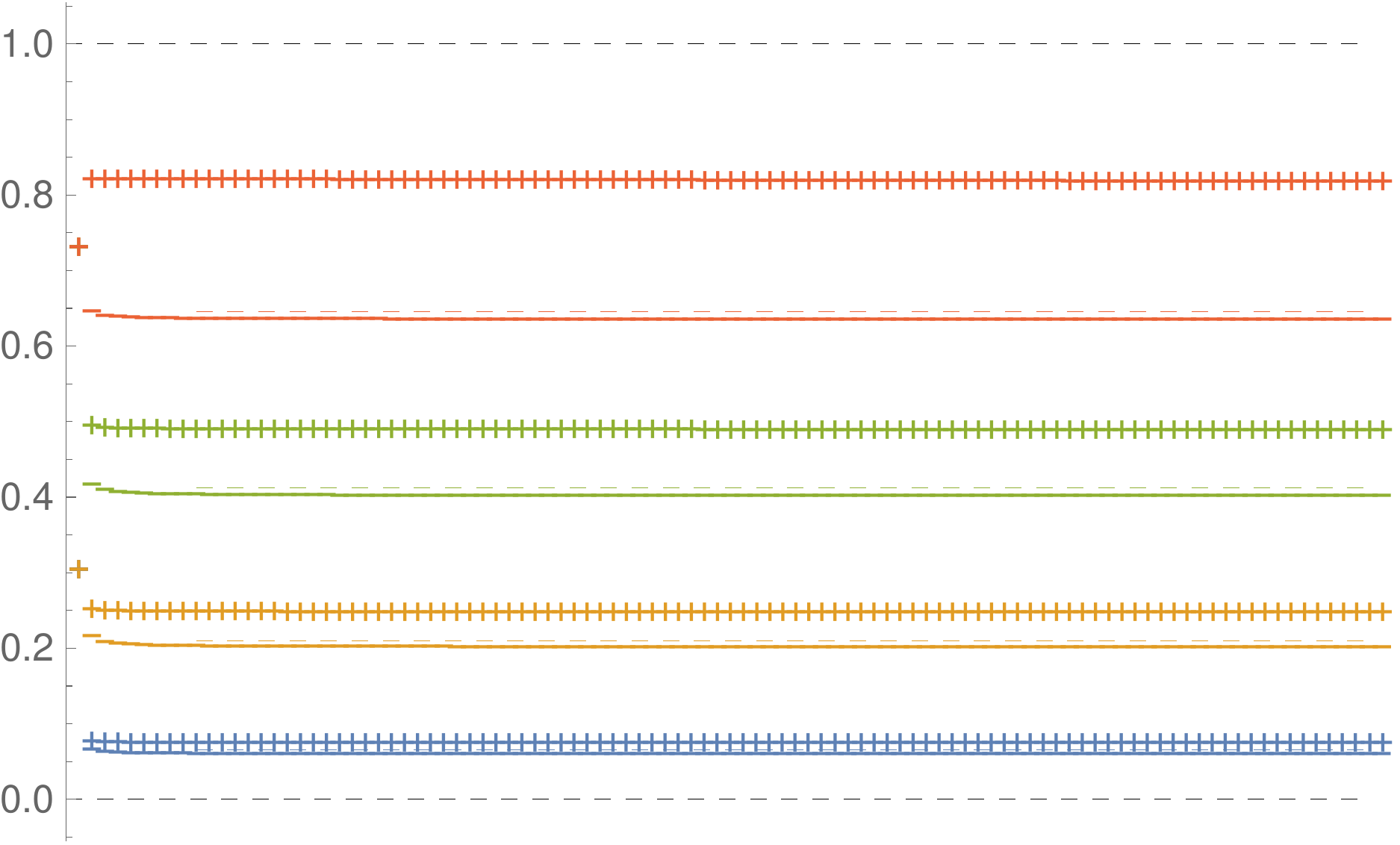}
First $10^5$ steps.
    \end{minipage}
\qquad
    \begin{minipage}{0.25\textwidth}
        \centering
        \includegraphics[width=\linewidth, 
        height=0.15\textheight]{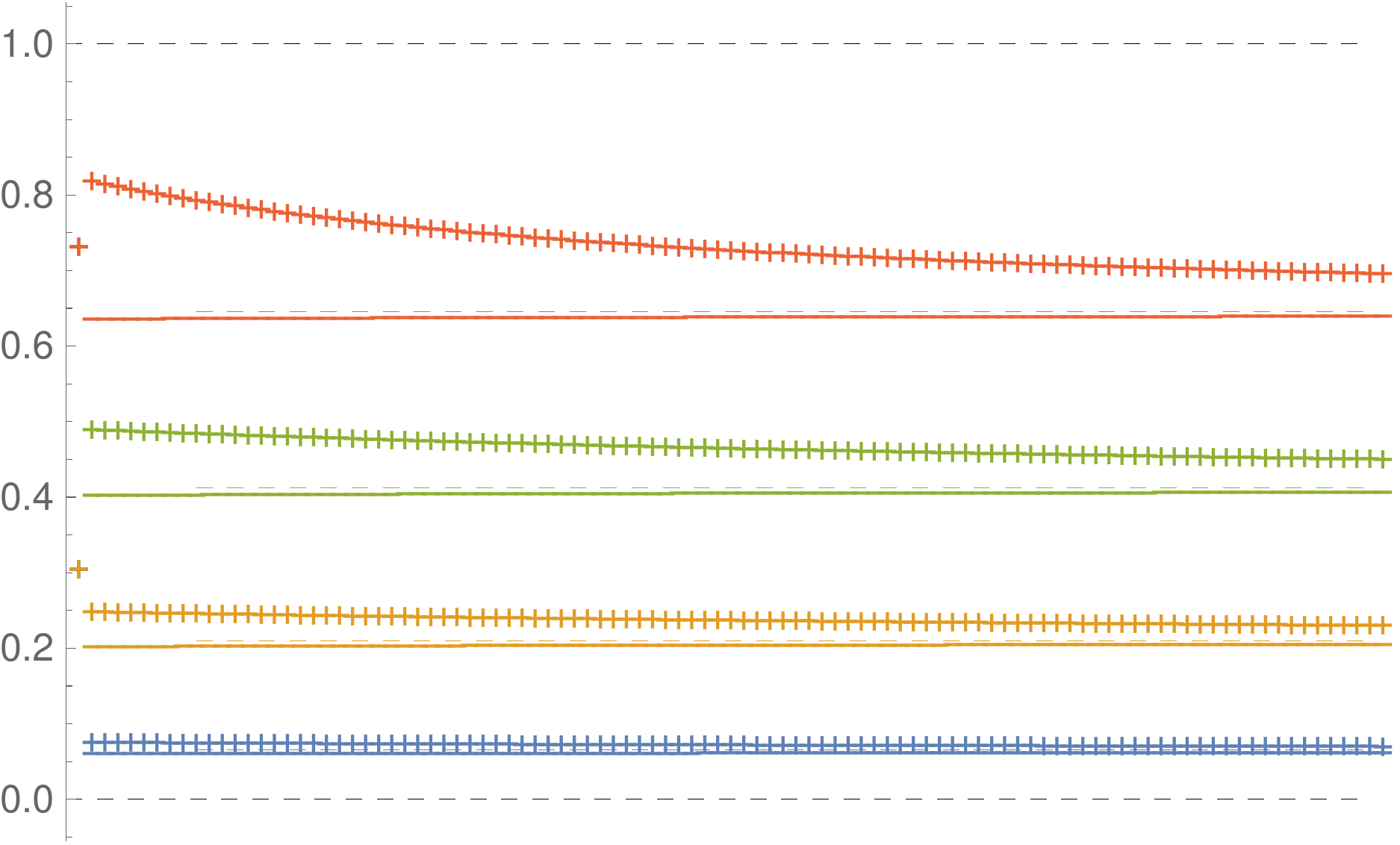}
First $10^7$ steps.
    \end{minipage}

\caption{Evolution of roots of $\tilde{p}_{A^i}$ (+) vs $\tilde{p}_{B^i}$ (-) 
vs asymptotes 
(dashed lines).}\label{fig:scales}
\end{figure}

For the second goal, we will focus on the part of the process where 
the most action happens.
Figure~\ref{fig:match} shows the first 100 steps of the process for each 
matrix 
process (in colors) as well as the paths taken by the roots of the 
$\tilde{p}_{M^i}$, 
in black).
The third plot shows the the two processes together.
Figure~\ref{fig:match} suggests the $\tilde{p}_{M^i}$ are capturing the 
general behavior of each $\overline{c}^2_{M^i}$ quite well. 
Furthermore, the distance between the two seems to be rather consistent.  
This is corroborated by Figure~\ref{fig:dist}, which shows the actual 
distributions of $c^2_{A^i}$ normalized so that the root of $\tilde{p}_{A^i}$ 
is at the 
center.

\begin{figure}[!h]
    \centering

    \begin{minipage}{0.25\textwidth}
        \centering
        \includegraphics[width=\linewidth, 
        height=0.15\textheight]{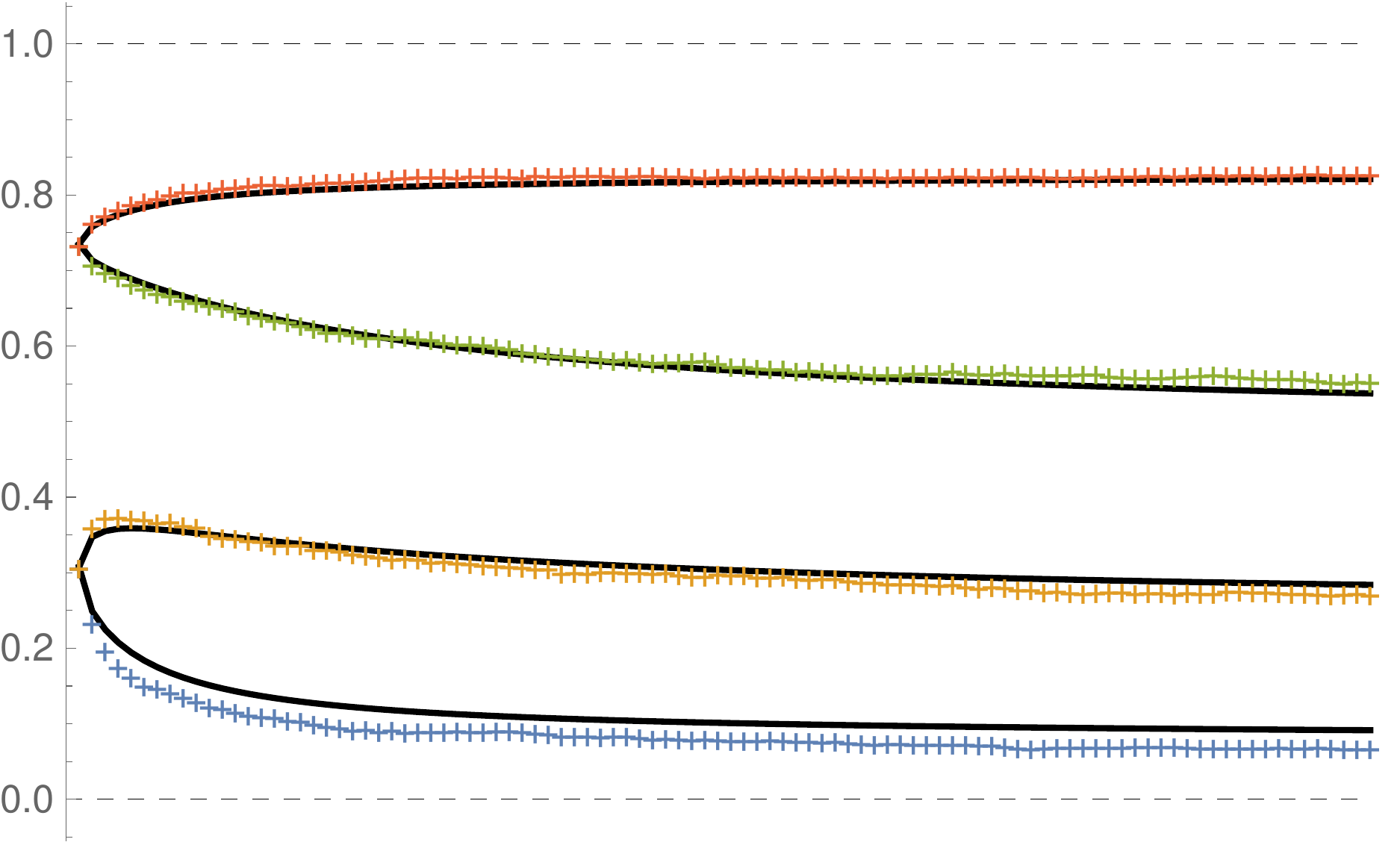}
$\overline{c}^2_{A^i}$ vs $\tilde{p}_{A^i}$
    \end{minipage}
\qquad
    \begin{minipage}{0.25\textwidth}
        \centering
        \includegraphics[width=\linewidth, 
        height=0.15\textheight]{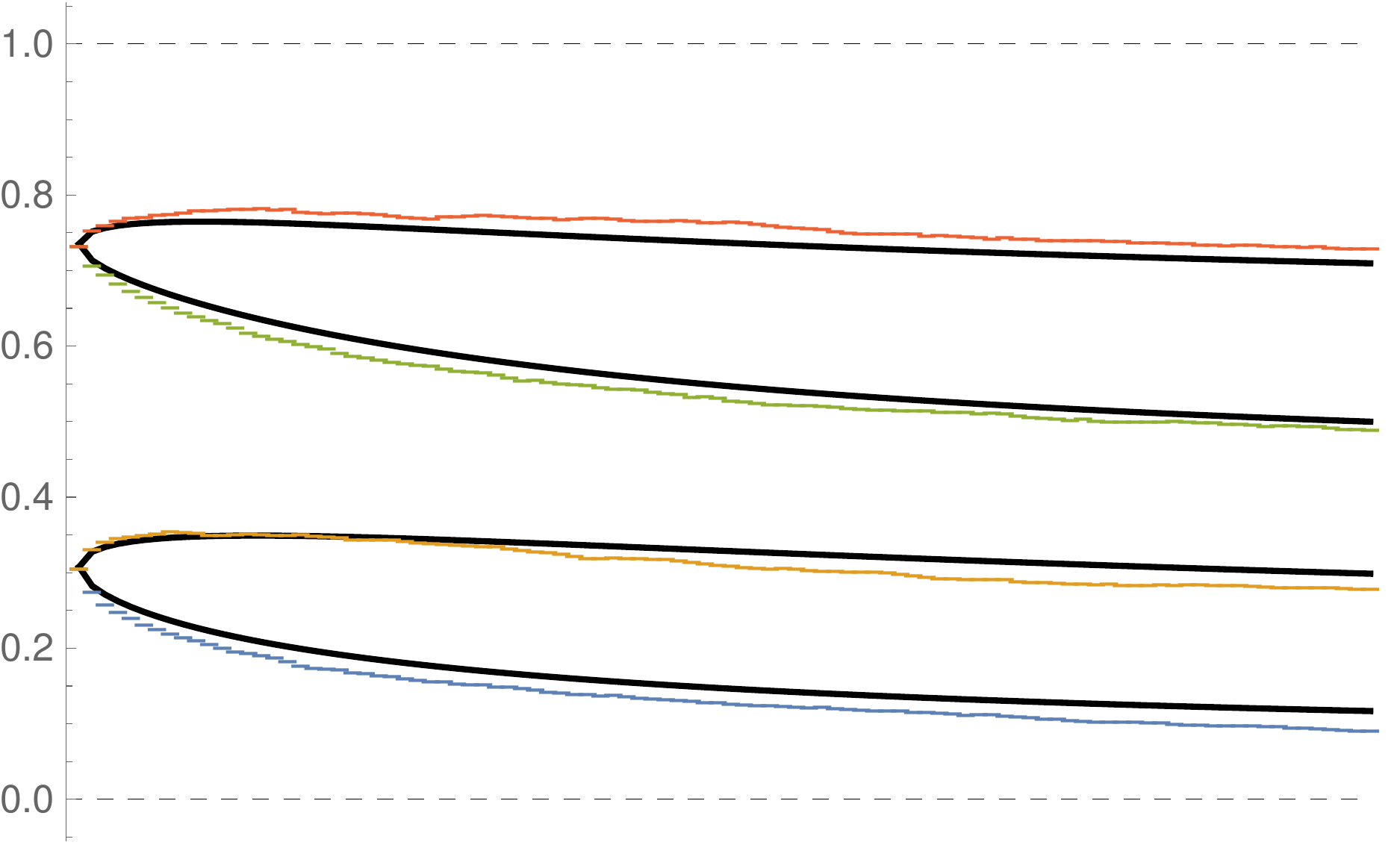}
$\overline{c}^2_{B^i}$ vs $\tilde{p}_{B^i}$
    \end{minipage}
\qquad
    \begin{minipage}{0.25\textwidth}
        \centering
        \includegraphics[width=\linewidth, 
        height=0.15\textheight]{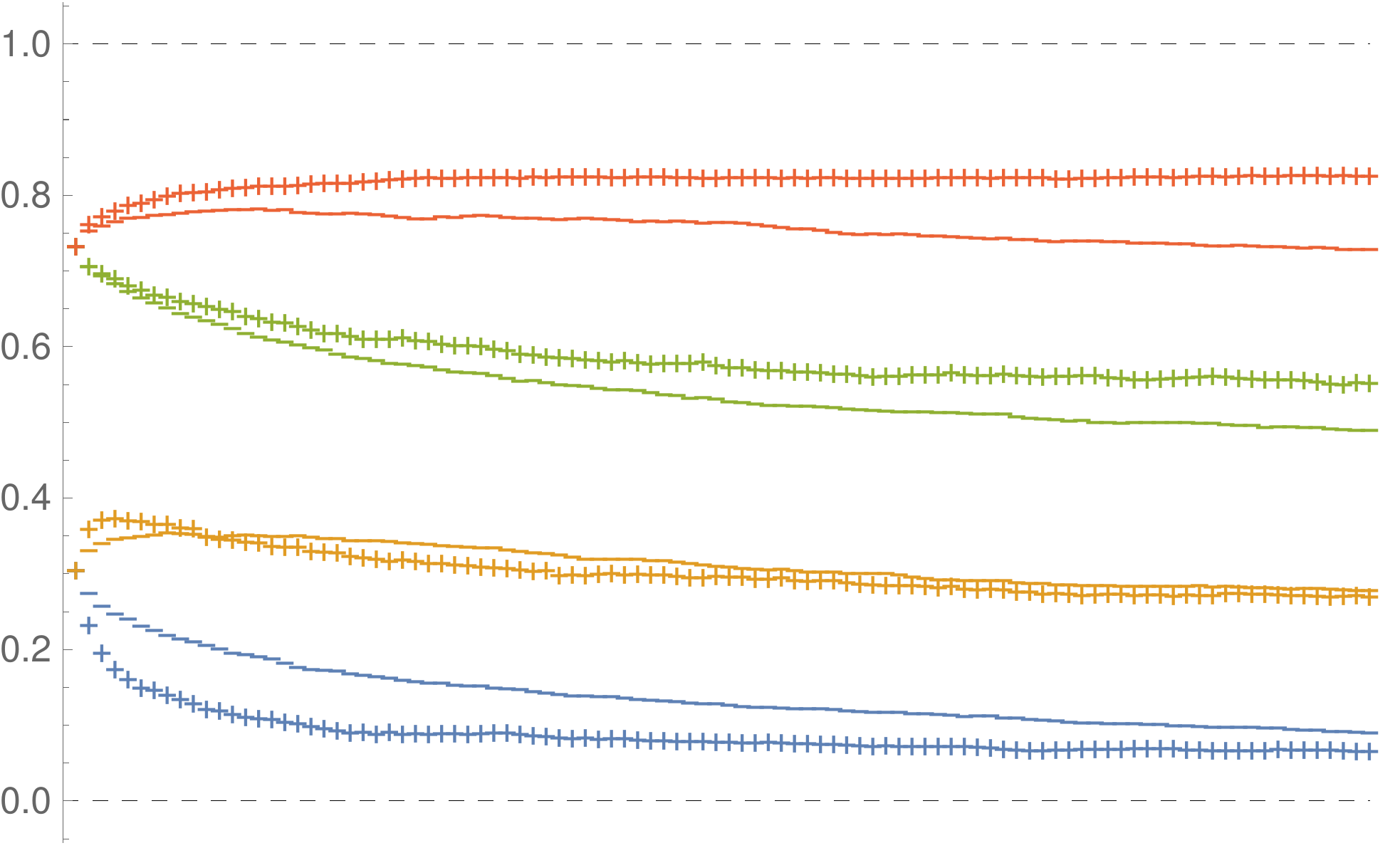}
$\overline{c}^2_{A^i}$ (+) vs $\overline{c}_{B^i}$ (-)
    \end{minipage}

 \caption{Experimental values (colors) versus polynomial 
 predictions (black) for the first 100 steps. }
\label{fig:match}
\end{figure}

\begin{figure}[!h]
    \centering

    \begin{minipage}{0.2\textwidth}
        \centering
        \includegraphics[width=\linewidth, 
        height=0.1\textheight]{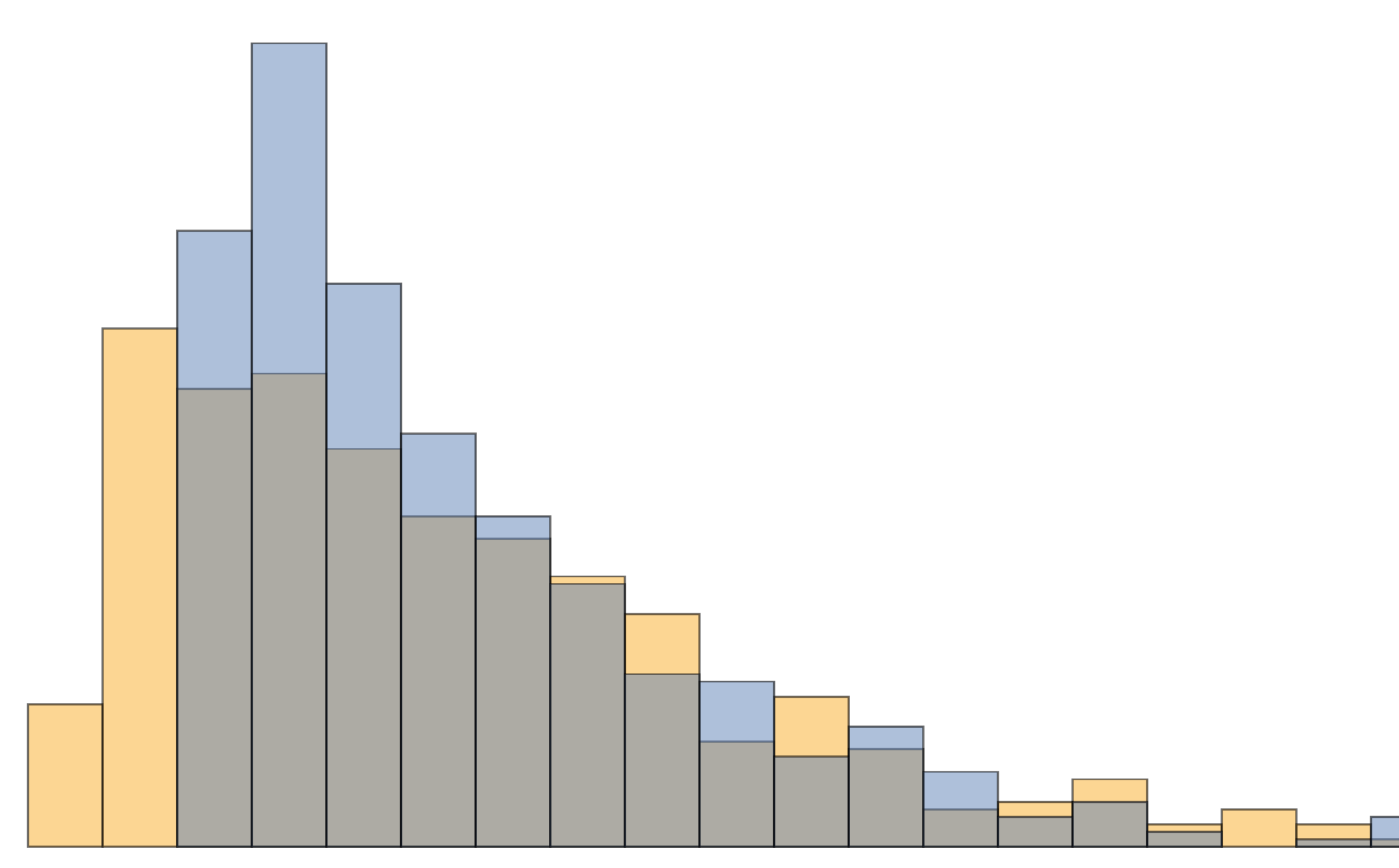}
blue (smallest)
    \end{minipage}
\qquad
   \begin{minipage}{0.2\textwidth}
        \centering
        \includegraphics[width=\linewidth, 
        height=0.1\textheight]{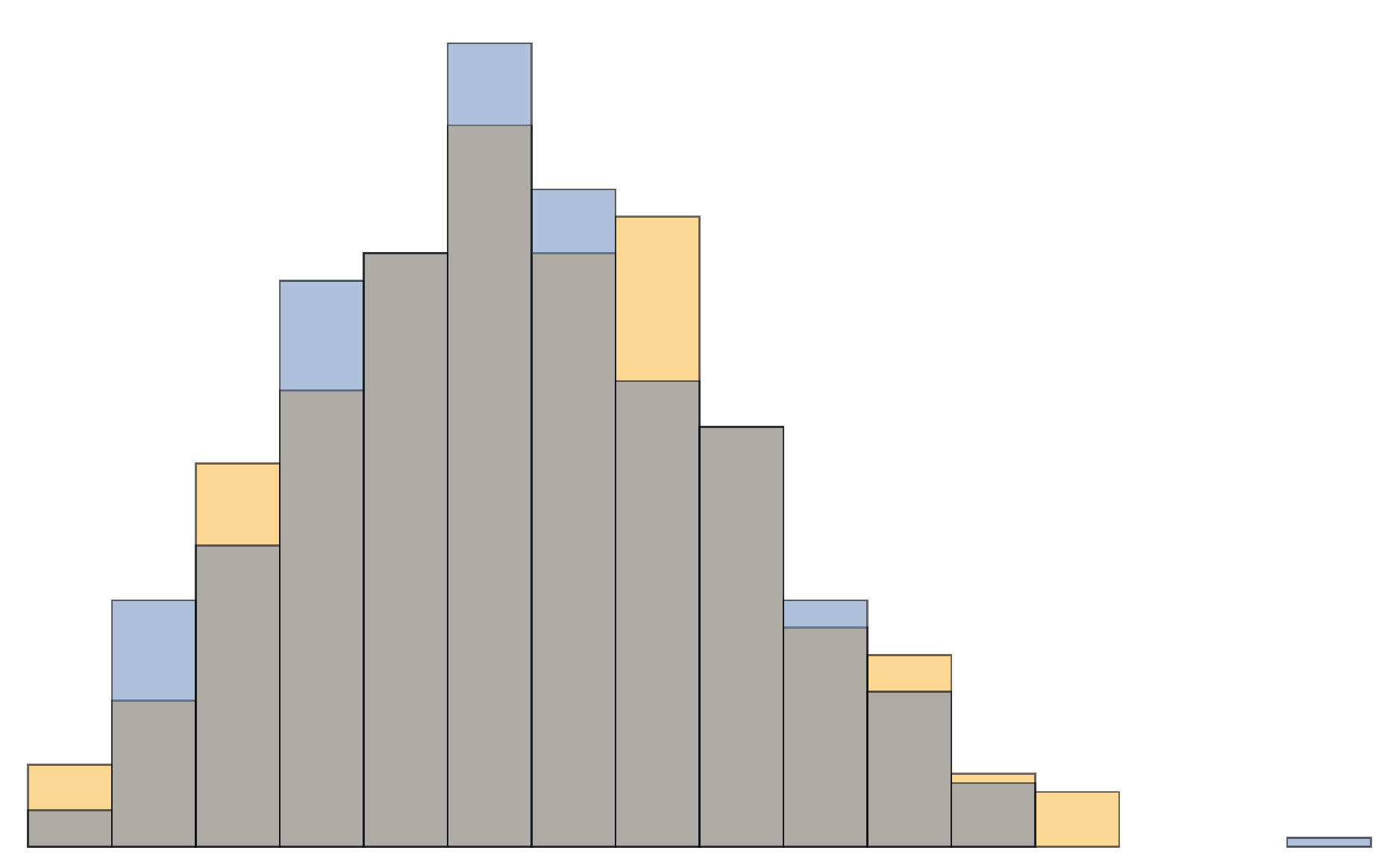}
yellow
    \end{minipage}
\quad
   \begin{minipage}{0.2\textwidth}
        \centering
        \includegraphics[width=\linewidth, 
        height=0.1\textheight]{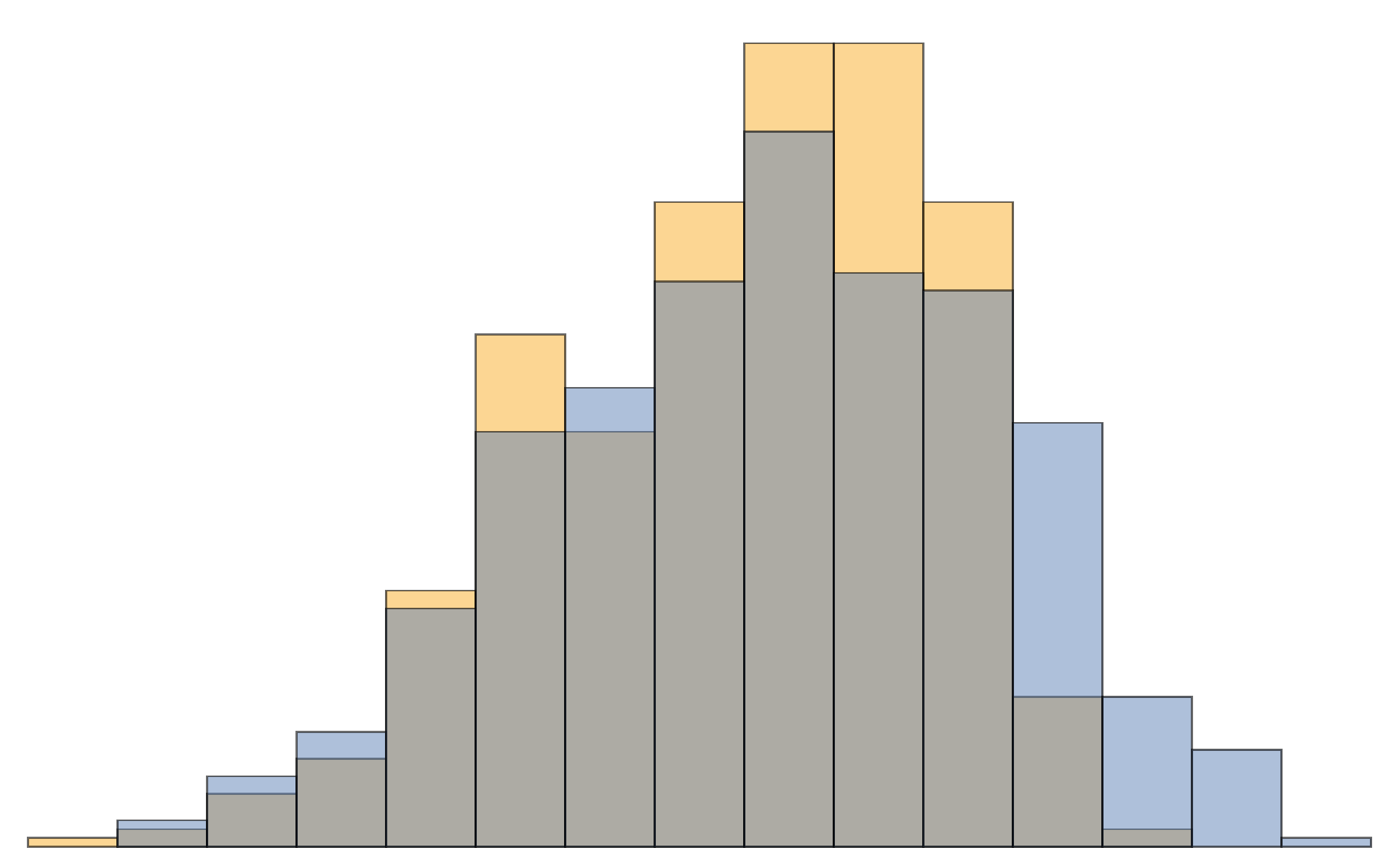} 
green
\end{minipage}
\quad
   \begin{minipage}{0.2\textwidth}
        \centering
        \includegraphics[width=\linewidth, 
        height=0.1\textheight]{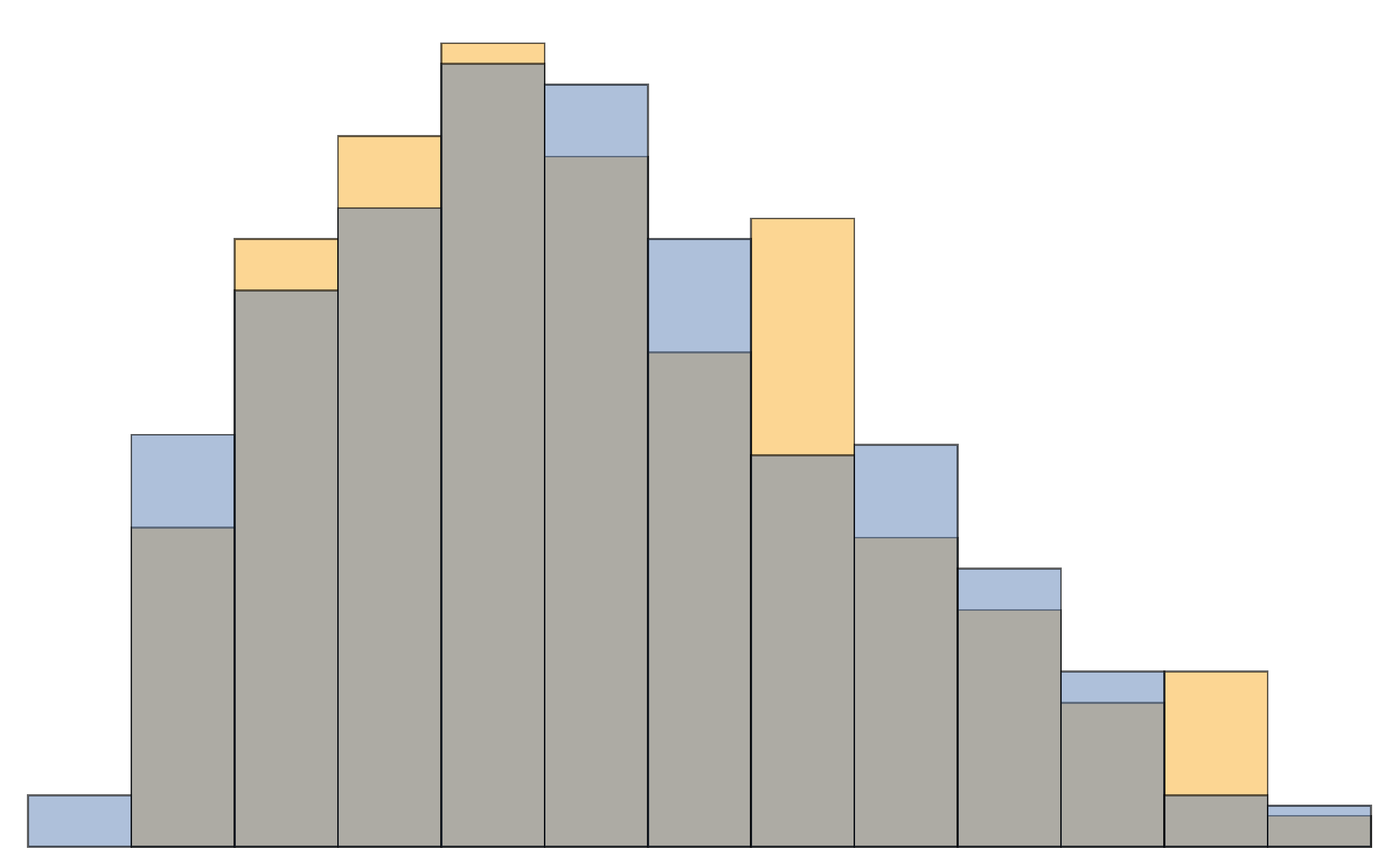}
red (largest)
    \end{minipage}
\caption{The distributions of $c^2_{A^i}$ shifted so that the 
centers are the roots of $\tilde{p}_{A^i}$ (for $i=20, 100$).}
\label{fig:dist}
\end{figure}
%
%
%
%

To address the third goal we will find it more useful to examine the moments 
induced by the $\overline{c}^2_{M^i}$ instead of the points themselves.
Figure~\ref{fig:moments1} and Figure~\ref{fig:moments2} show the relationship 
between the moments generated by the $\overline{c}_{M^i}$ (in yellow) and the 
moments generated by the roots of the $\tilde{p}_{M^i}$ (in blue).
An additional data point (in green) consists of the moments of the 
$\overline{c}^2_{M^i}$ when we change the simulation slightly to use {\em 
complex} Gaussians when forming the random $Z_i$ (all transposes become 
conjugate transposes as well).
Note that these are the moments of $\overline{c}_{M^i}$ and not 
$\overline{c}^2_{M^i}$, so the first image is (in some sense) the $1/2$ moment 
of $\overline{c}^2_{M^i}$.
We found this view to be more compelling as it shows the concave functions 
monotone increasing in $\beta$ and the convex ones monotone decreasing (signs 
of possible majorization)\footnote{One might also notice in  
Figure~\ref{fig:match} that the errors in the largest value (red) are 
consistently positive, whereas the errors in the smallest value (blue) are 
negative (also signs of possible majorization).}.
We will discuss conjectures related to this relationship in 
Section~\ref{sec:conclusion}.

\begin{figure}[!h]
    \centering
    \begin{minipage}{0.2\textwidth}
        \centering
        \includegraphics[width=\linewidth, 
        height=0.1\textheight]{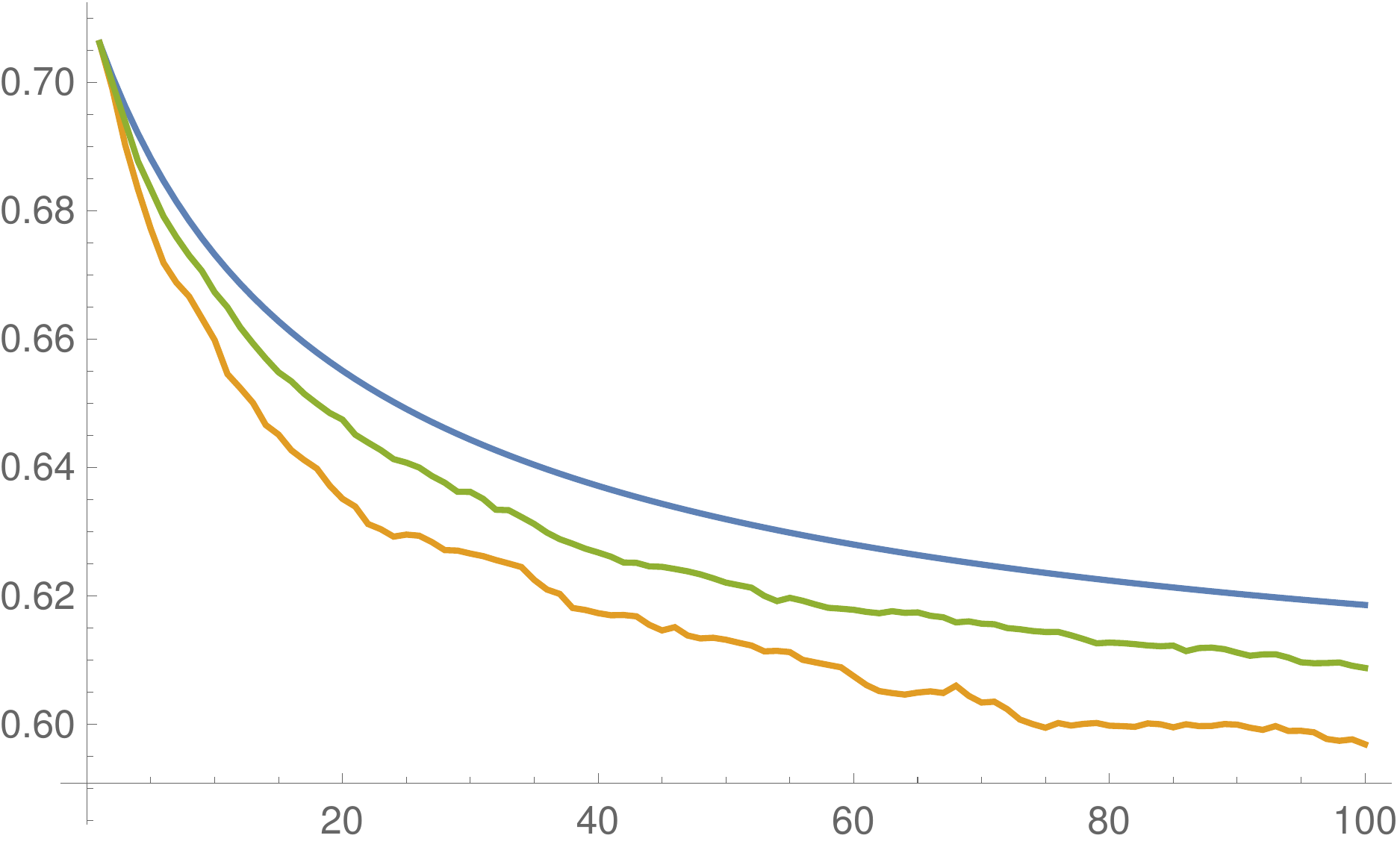}
    \end{minipage}
\qquad
   \begin{minipage}{0.2\textwidth}
        \centering
        \includegraphics[width=\linewidth, 
        height=0.1\textheight]{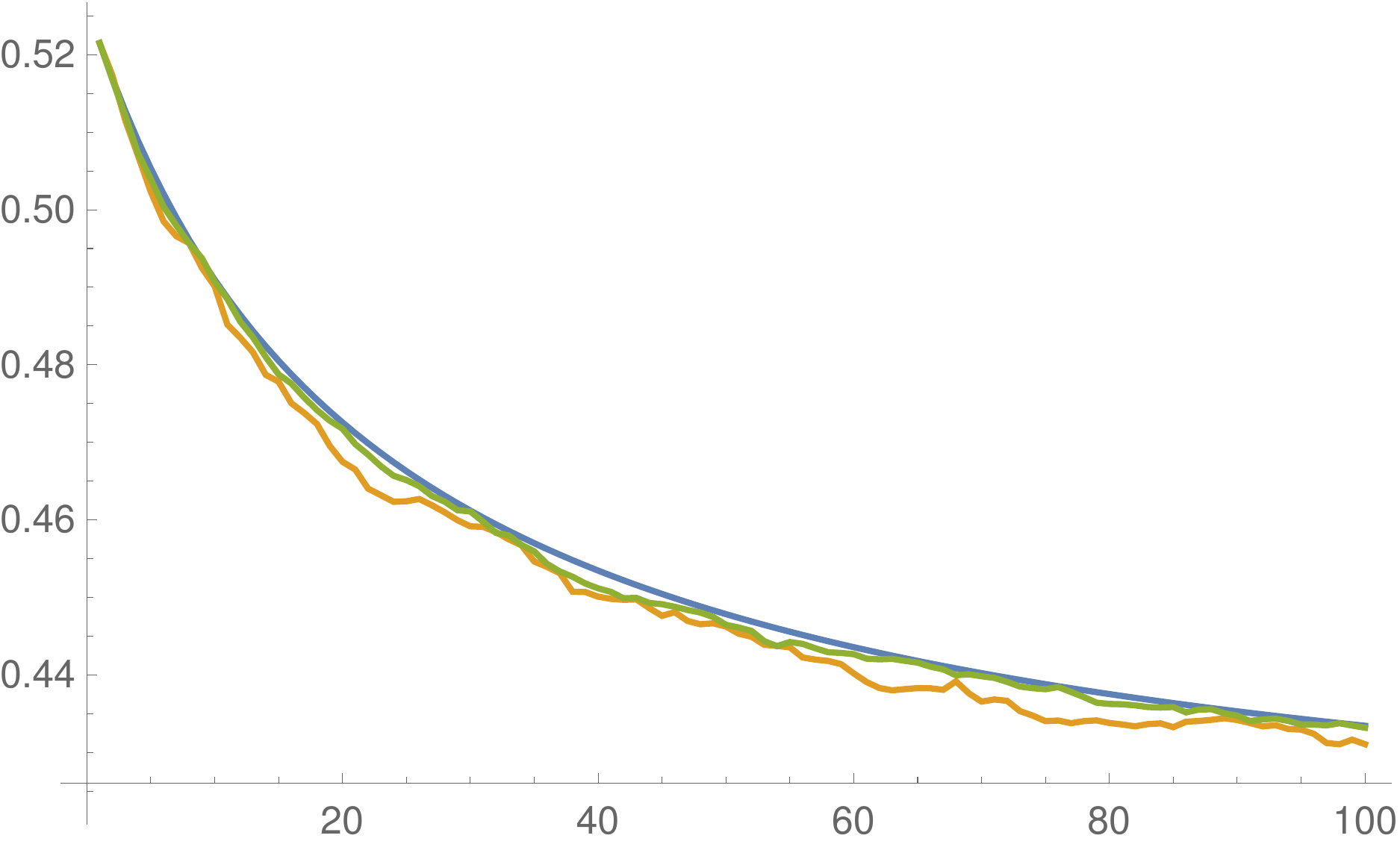}
    \end{minipage}
\qquad
  \begin{minipage}{0.2\textwidth}
        \centering
        \includegraphics[width=\linewidth, 
        height=0.1\textheight]{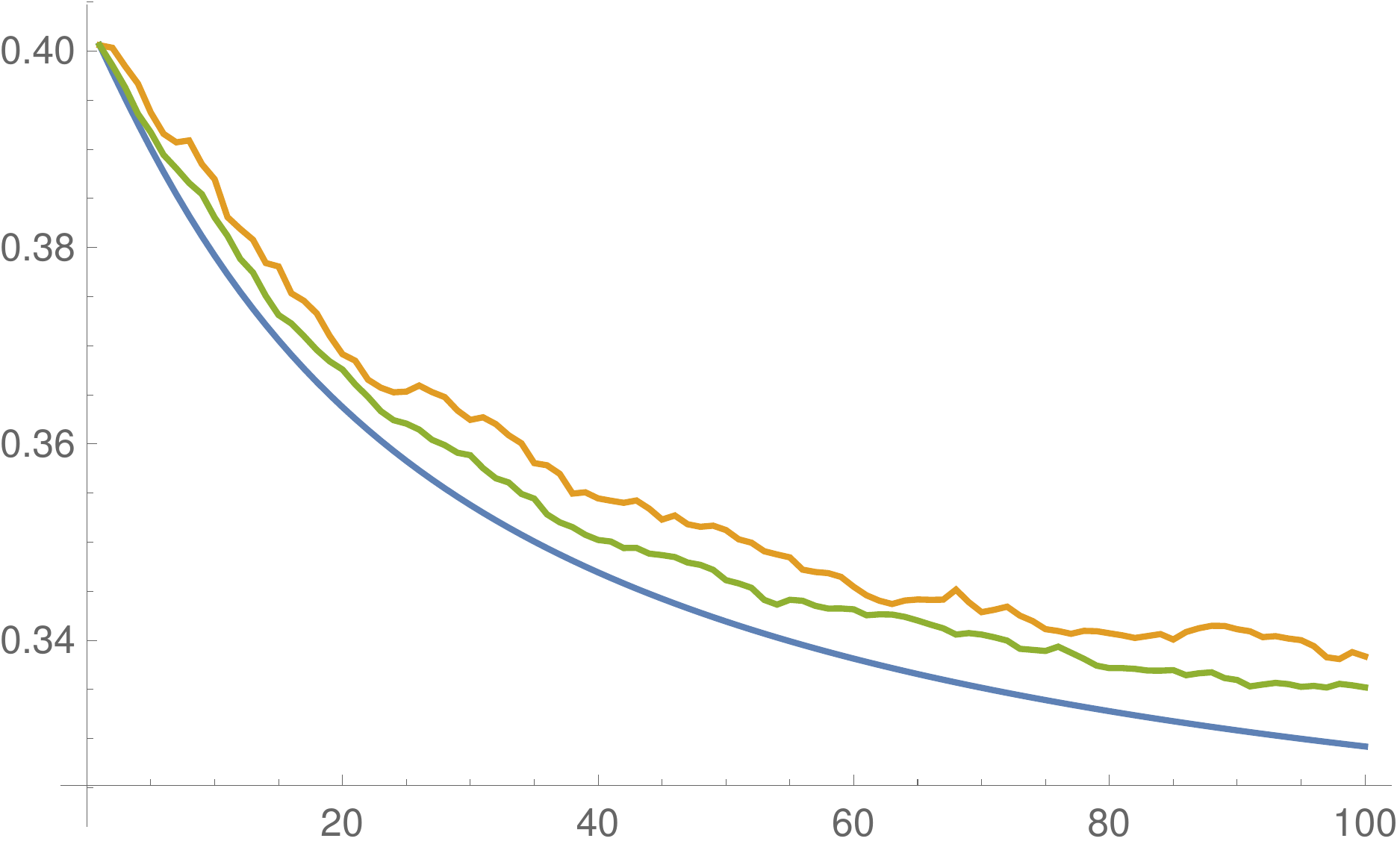}
    \end{minipage}
\qquad
   \begin{minipage}{0.2\textwidth}
        \centering
        \includegraphics[width=\linewidth, 
        height=0.1\textheight]{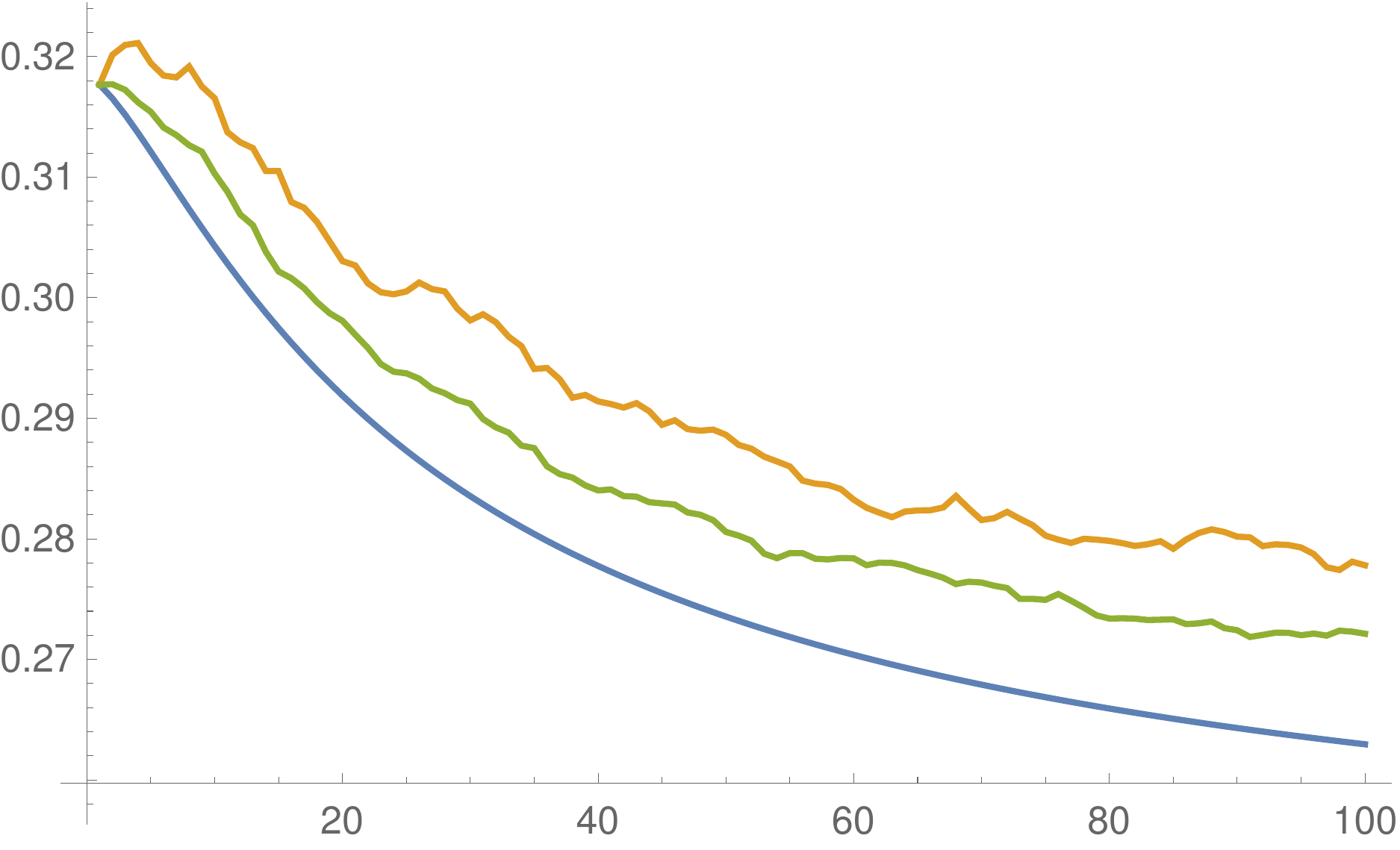}
    \end{minipage}
\caption{ The first four moments derived from $\overline{c}_{A^i}$ for $Z_i$ 
real valued (yellow) and complex valued (green).
The moments derived from the roots of $\tilde{p}_{A^i}(0, w-1, w)$ are in blue.}
\label{fig:moments1}
\end{figure}

\begin{figure}[!h]
    \centering
    \begin{minipage}{0.2\textwidth}
        \centering
        \includegraphics[width=\linewidth, 
        height=0.1\textheight]{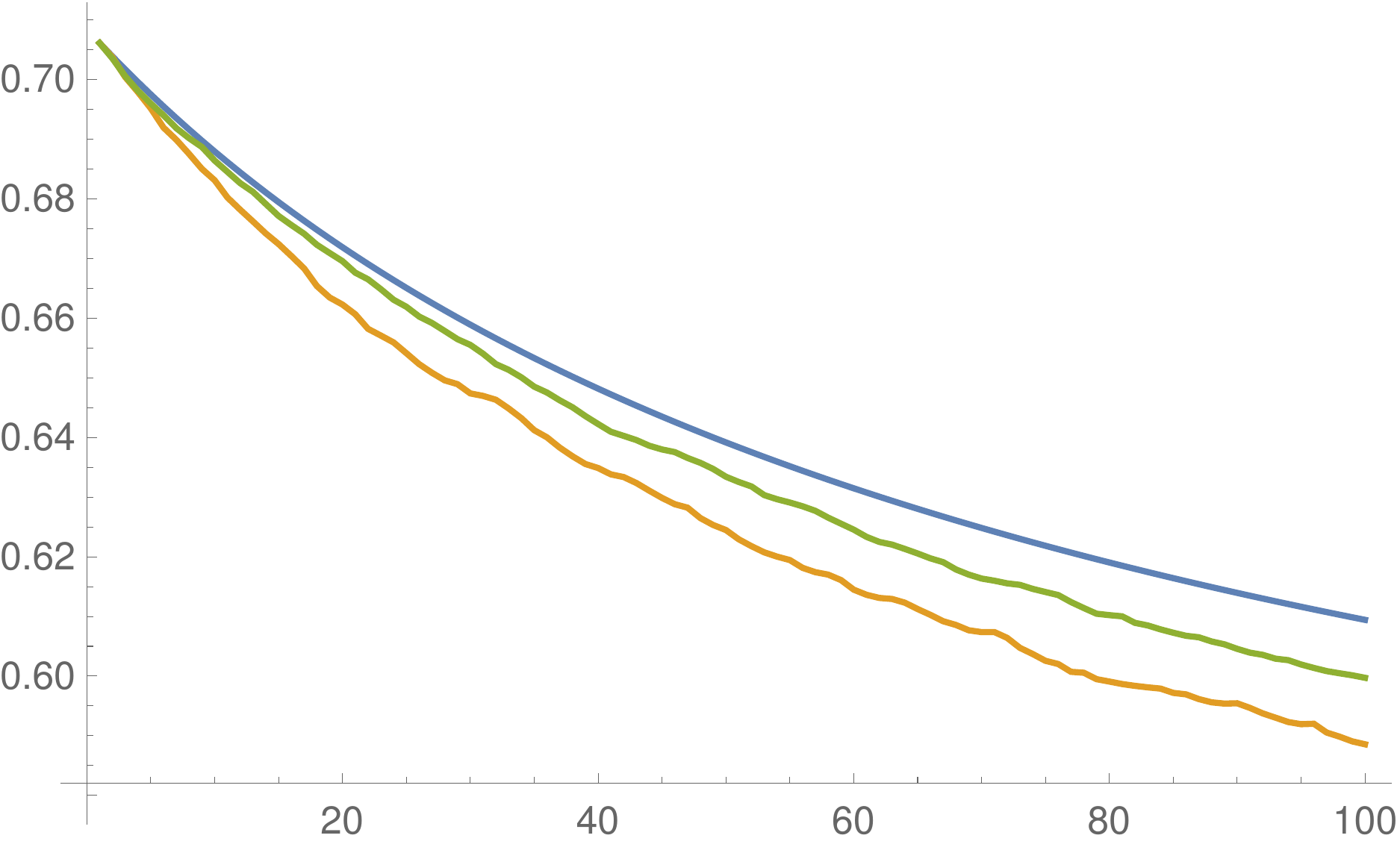}
    \end{minipage}
\qquad
   \begin{minipage}{0.2\textwidth}
        \centering
        \includegraphics[width=\linewidth, 
        height=0.1\textheight]{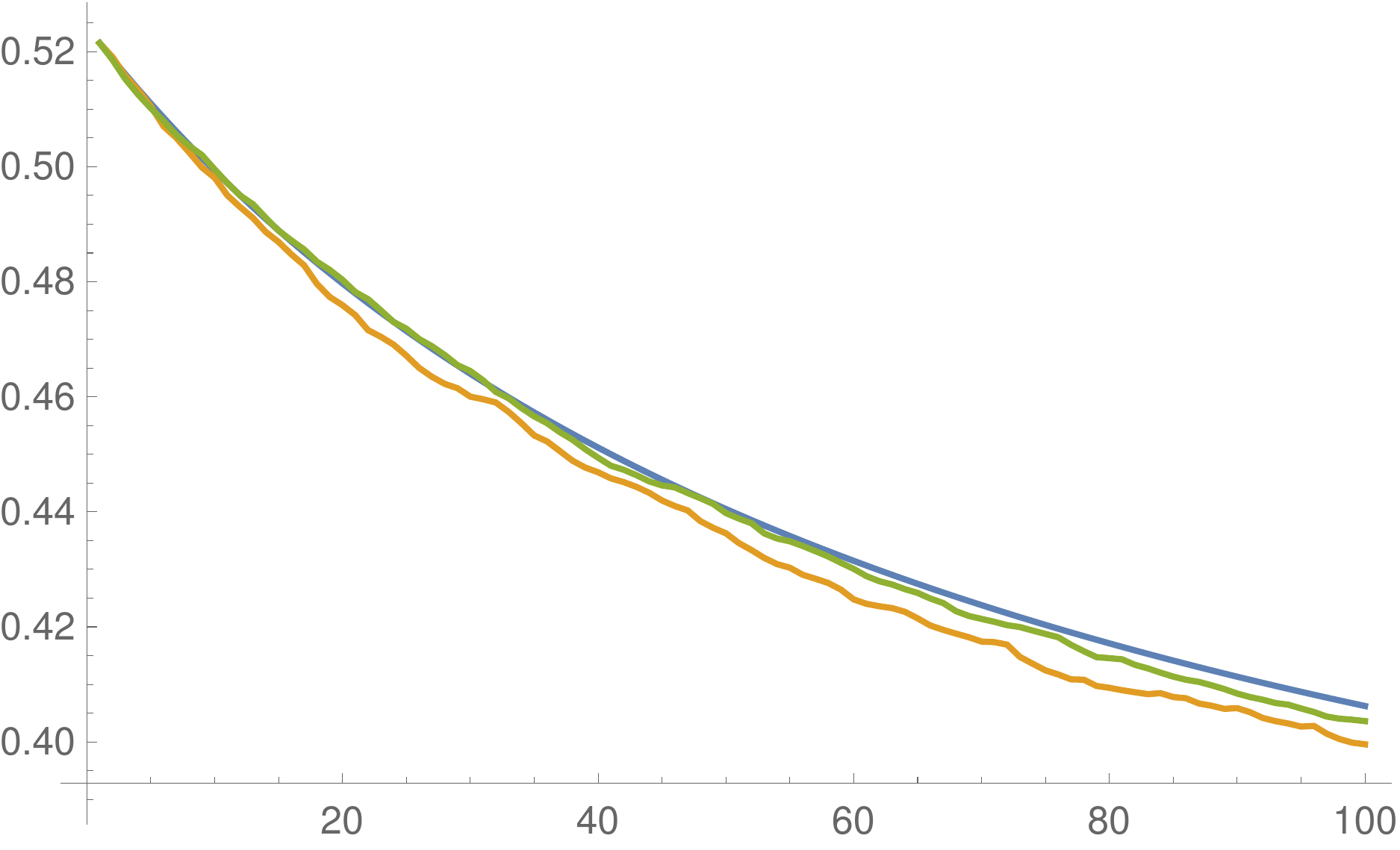}
    \end{minipage}
\qquad
  \begin{minipage}{0.2\textwidth}
        \centering
        \includegraphics[width=\linewidth, 
        height=0.1\textheight]{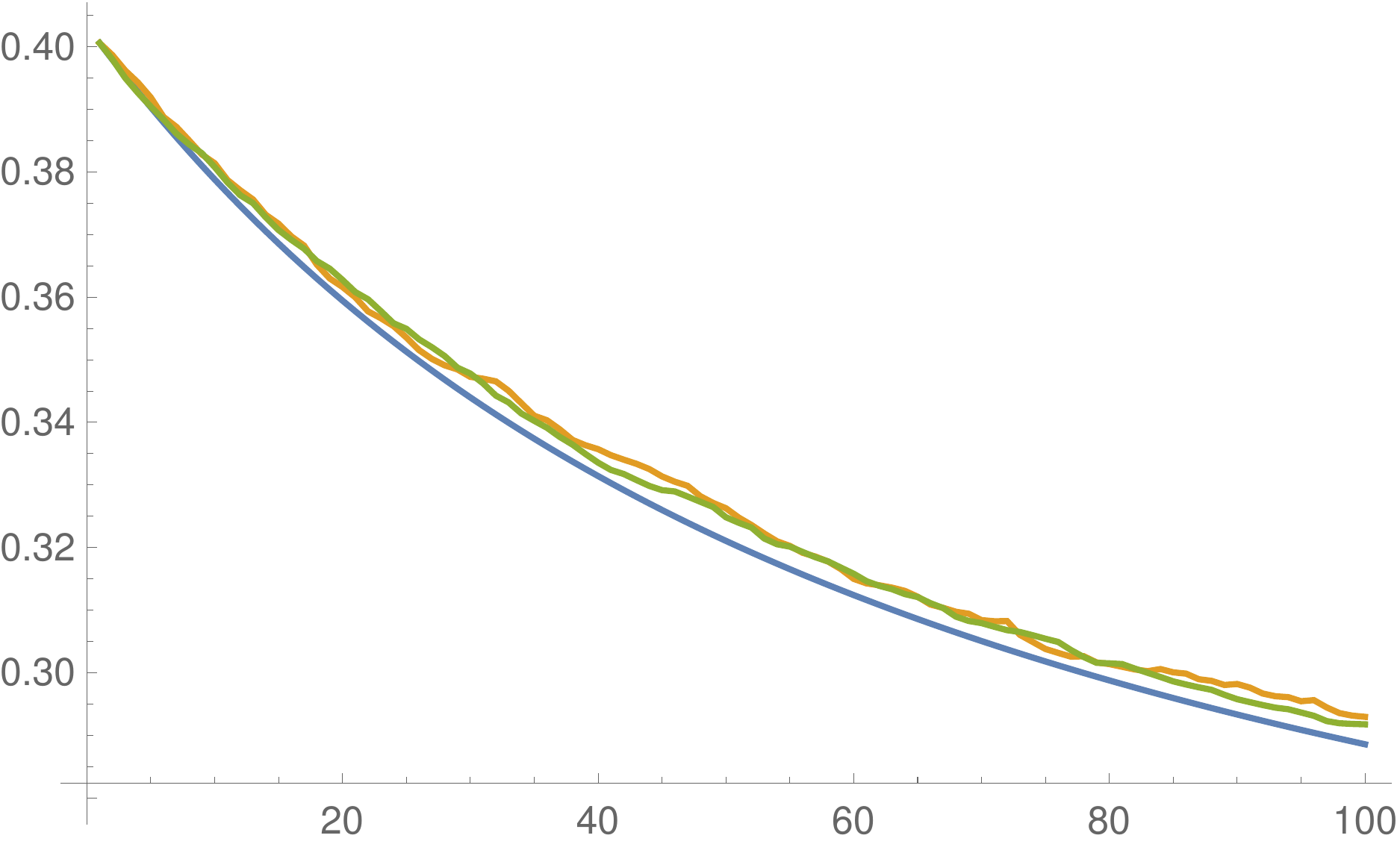}
    \end{minipage}
\qquad
   \begin{minipage}{0.2\textwidth}
        \centering
        \includegraphics[width=\linewidth, 
        height=0.1\textheight]{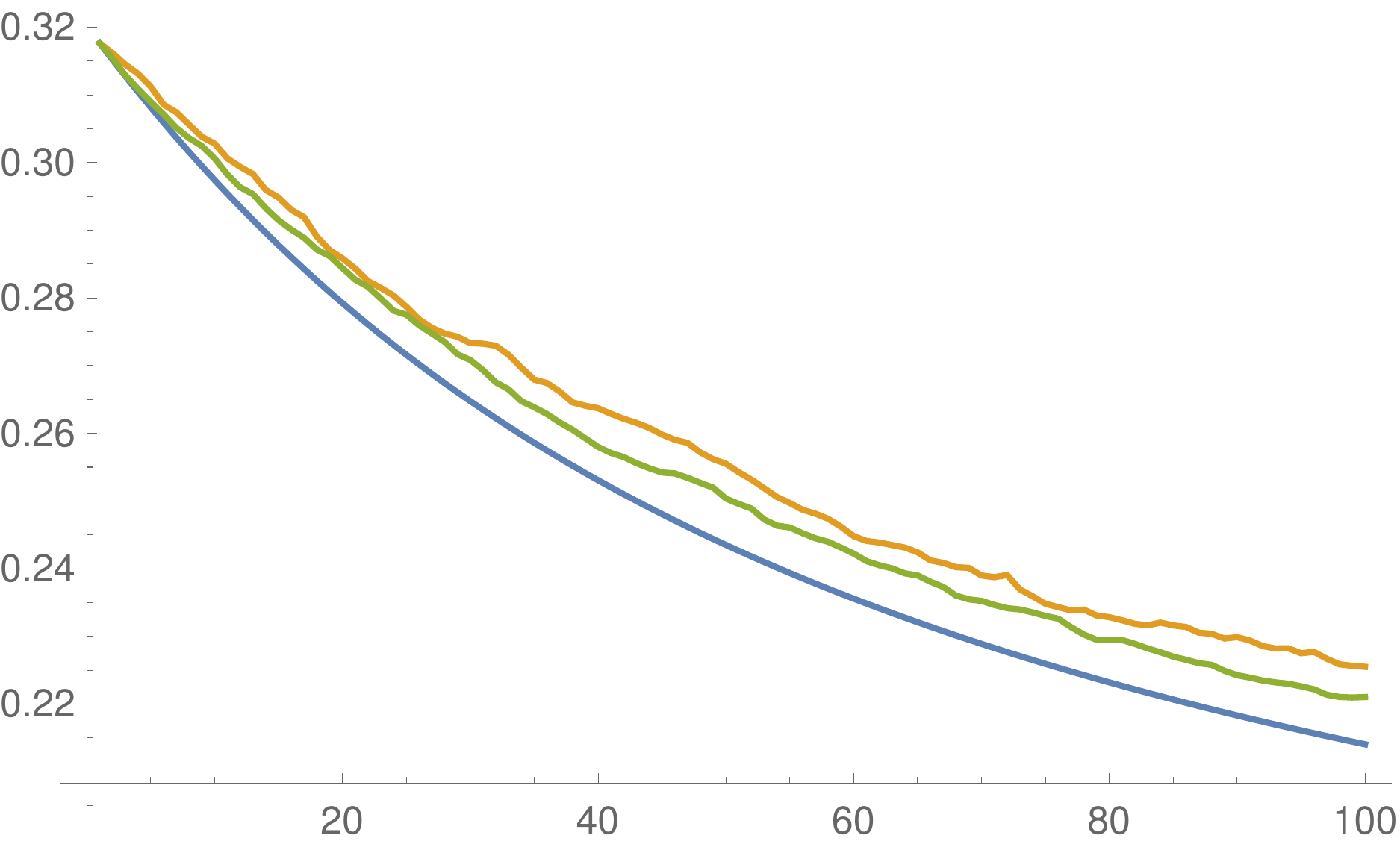}
    \end{minipage}
\caption{ The first four moments derived from $\overline{c}_{B^i}$ for $Z_i$ 
real valued (yellow) and complex valued (green).
The moments derived from the roots of $\tilde{p}_{B^i}(0, w-1, w)$ are in blue.}
\label{fig:moments2}
\end{figure}

\section{Conclusion}\label{sec:conclusion}

In many ways, it seems that the work in this paper suggests more interesting 
problems than it solves.
A number of these questions were discussed at the beginning 
of Section~\ref{sec:simulation}, however we would like to comment a bit more on 
the theme of the polynomial convolution matching the ``non-random'' part of 
some random matrix process. 
One intriguing part of the polynomial convolutions is that (unlike many other 
things in random matrix theory) they do not depend on $\beta$, a fact that has 
been discovered in various contexts a number of times \cite{B2, vadim}.
It is also quite common for $\beta$-ensembles on $\nn \times \nn$ matrices to 
converge as $\beta \to \infty$ to a uniform distribution on a set of $\nn$ 
points.
Any such distribution (finite, on $\nn$ points) is completely determined by the 
values of its expected elementary symmetric polynomials, and so one can hope to 
find a polynomial convolution which captures this behavior completely.
The typical way to prove this, however, would be to express the behavior of 
interest as a function of $\beta$ and then take the limit.
In situations where it is computationally infeasible to find such a function, 
are there other methods one could use to prove such a correspondence?
If so, Section~\ref{sec:simulation} suggests that there may be a 
measurable relationship between distributions as $\beta$ increases, and so such 
a result could lead to improved estimates for classical ensembles.

\subsection{Random matrix theory}

The most obvious question in this direction is whether a Gaussian point 
process that comes from  \textbf{4.} can be solved.  
That is, given matrices $A, G \in \MM_{\mm, \nn}$ where $A$ is fixed and the 
entries of $G$ are i.i.d. Gaussians, can we 
find the exact distribution of the (squares of the) generalized singular values 
of $A + \theta G$?
It is certainly understandable if any previous attempt seemed overly daunting 
--- as we have mentioned, instantaneous behavior of the Hermite, Laguerre, and 
Jacobi matrix processes depend only on the current point configuration and a 
small number of parameters, whereas it should 
be clear from Section~\ref{sec:simulation} that the point process derived from  
\textbf{4.} depends on a much larger number of parameters.
However the results of Section~\ref{sec:gsvdbm} suggest that these parameters 
can be captured by natural relationships between the two matrices, in which 
case an explicit formula could be possible.
We would think that such a formula would certainly be an interesting 
development in field.


\subsection{Finite free probability}

Finding a polynomial convolution for a matrix operation (when it exists) tends 
to be a fairly reasonable task; the results of \cite{lg}, for example, reduce a 
number of possible convolution combinations to a straightforward calculation.
The opposite direction --- given an operation on polynomials, trying to find 
matrices (and matrix operations) that they correspond to --- seems much harder.
Even in the case where we know the operation, it tends to be hard to prove that 
there exist concrete (non-random) matrices that behave in this way.
The major tool in this respect is a result of Helton and Vinnikov concerning 
real stable polynomials \cite{hv}, however their result is quite quickly 
reaching the end of its utility (it is not true when the homogeneous 
polynomials in question have more than $3$ variables).
That said, it would not be surprising if polynomial convolutions that are based 
on matrix operations did not also maintain determinantal representation, and it 
would be interesting to find new ways to prove such statements.

The condition of the random matrix $X_{\mm, \nn}$ being symmetric turns out to 
be far more than we need --- as we have seen, the expected characteristic 
polynomial (\ref{eq:xyzcp}) is (at most) quadratic in the entries of the 
individual matrices.
The utility of having the symmetry condition is that it makes the resulting 
random matrix distribution $\SPM$--bi-invariant.
This suggests that a weaker condition than $\SPM$--bi-invariance might be 
sufficient for gaining the required amount of symmetry to be able to use 
Theorem~\ref{thm:lg}.
One natural candidate to replace the signed permutation matrices is the 
collection of matrices in the standard representation of $S_{n+1}$ (a set of 
size $(n+1)!$ instead of $2^n n!$). 
Furthermore, the validity of this replacement would follow easily from a 
conjecture in \cite{lg} that Theorem~\ref{thm:lg} holds in a slightly more 
general context (we refer the interested reader to \cite{lg} for more details).

\subsection{Acknowledgments}\label{sec:thanks}

The author would like to recognize the IPAM program in Quantitative Linear 
Algebra, without which this paper would likely not exist.
The results presented here all came, in one way or another, from discussions 
that were started at IPAM.
The author would also like to thank Benno Mirabelli, who helped to continue 
these discussions after the program had ended (and eventually led to the 
ideas in this paper).


\begin{thebibliography}{999}

%
%


\bibitem{biane} P. Biane. Free Brownian motion, free stochastic calculus and 
random matrices. Fie. Inst. Comm. 12, Amer.
Math. Soc. Providence, RI, 1997



\bibitem{B2} A. Borodin, Stochastic higher spin six vertex model and Macdonald 
measures. J. of Math. Physics 59.2 (2018): 023301.



\bibitem{bru} M. F. Bru, Wishart Process. J. Theor. Prob. 4, (1991) 725–751. 
MR113213


\bibitem{collins} B. Collins, Int\'{e}grales matricielles et probabilit\'{e}s 
non-commutatives, Ph.D. thesis, Universit\'{e} Paris 6, 2003.

\bibitem{demni} N. Demni. Free Jacobi process. J. of Theor. 
Prob. 2008 Mar 1; 21(1): 118-143.

\bibitem{doumerc} Y. Doumerc, Matrices, al\'{e}atoires, processus stochastiques 
et groupes de r\'{e}flexions, Ph.D. thesis, Universit\'{e} Toulouse III.

\bibitem{dyson} F. J. Dyson, A Brownian-motion model for the eigenvalues of a 
random matrix, J. Math. Phys.
3 1191-1198 (1962)

\bibitem{edelman_thesis} A. Edelman, Eigenvalues and Condition Numbers of Ranom 
Matries.
Thesis.


\bibitem{edelmanCS} A. Edelman, B. D. Sutton, The Beta-Jacobi Matrix Model, the 
CS Decomposition, and Generalized Singular Value Problems.
Found. of Comp. Math., Vol 8 (2008), pp. 259--285.


\bibitem{forrester} P. J. Forrester, Log-gases and random matrices (LMS-34). 
Princeton University Press. 2010.

\bibitem{godsil} C. D. Godsil. Algebraic combinatorics. Chapman \& Hall/CRC, 
1993.

\bibitem{gg} C. D. Godsil, I. Gutman, On the matching polynomial of 
a graph. 
In L. Lov\'{a}sz and V. T. S\'{o}s, ed., Algebraic Methods in graph theory, 
Vol. I of Colloq. Math. Soc. J\'{a}nos Bolyai, 25, pages 
241-249. J\'{a}nos Bolyai Math. Soc., 1981.

\bibitem{golub} G. H. Golub and C. F. Van Loan, Matrix computations, 4th 
Edition, JHU Press, 2013.

\bibitem{vadim} V. Gorin, A. W. Marcus, Crystallization of random matrix 
orbits. 
Intern. Math. Res. Notices 2020.3 (2020): 883-913.

\bibitem{aurelien_ff} A. Gribinski, A. W. Marcus, Existence and polynomial time 
construction of biregular, bipartite Ramanujan graphs of all degrees, 
\arxiv{2108.02534}.



%
\bibitem{hv}  W.  Helton, V. Vinnikov, Linear  matrix  inequality  
representation  of  sets.  Comm. on Pure and Appl. Math., 60(5), 2007: 
654-–674.
%
%
%

\bibitem{lg} A. W. Marcus, A class of multivariate polynomial 
convolutions (and applications), \arxiv{2010.08996}

\bibitem{DUI} A. W. Marcus, Discrete Unitary Invariance, \arxiv{1607.06679}

\bibitem{ff_main} A. W. Marcus, Polynomial convolutions and (finite) free 
probability, preprint.

%
%
%
%
%

\bibitem{mehta} Mehta, M. L., Random Matrices. 3rd ed. Vol. 142. Amsterdam: 
Elsevier Academic, 2004.

\bibitem{benno} B. P. Mirabelli, Hermitian, Non-Hermitian and
Multivariate Finite Free Probability, Ph.D. Thesis (2020).
%

\bibitem{tt} T. Tao, public communication.\\ \url{ 
https://terrytao.wordpress.com/2017/10/17/heat-flow-and-zeroes-of-polynomials/}

\bibitem{vanloan} C. F. Van Loan, Generalizing the Singular Value 
Decomposition. SIAM J. Numer. Anal. 13.1 (1976): 76–83.

\bibitem{szego} G. Szeg\H{o}. Orthogonal polynomials. Amer. Math. Soc.; 1939.

\end{thebibliography}
\end{document}